\documentclass[reqno]{amsart}

\usepackage{amsmath, amssymb, amsthm}
\usepackage{graphicx}
\usepackage{epstopdf}
\usepackage{caption}
\usepackage{appendix}
\usepackage{subcaption}
\usepackage{color}
\usepackage{hyperref}
\usepackage{tabularx}
\usepackage{enumerate}
\usepackage[utf8]{inputenc}
\usepackage[nocompress]{cite}
\usepackage{amsaddr}

\newcommand{\Bea}{\begin{eqnarray*}}
	\newcommand{\Eea}{\end{eqnarray*}}
\newcommand{\bea}{\begin{eqnarray}}
\newcommand{\eea}{\end{eqnarray}}

\newcommand{\Prob}{\mathbb{P}}

\newcommand{\C}{\mathbb{C}}
\newcommand{\D}{\mathbb{D}}
\newcommand{\E}{\mathbb{E}}
\newcommand{\N}{\mathbb{N}}
\newcommand{\R}{\mathbb{R}}

\newcommand{\Pb}{\mathbb{P}}

\newcommand{\CE}{\mathcal{E}}
\newcommand{\CF}{\mathcal{F}}

\newcommand{\CH}{\mathcal{H}}

\newcommand{\CS}{\mathcal{S}}

\DeclareMathOperator{\dist}{dist}

\newcommand{\comment}[1]{}

\def\1{\textbf{1}}

\DeclareMathOperator{\rank}{rank}
\DeclareMathOperator{\tr}{tr}

\newtheorem{thm}{Theorem}[section]
\newtheorem{prop}[thm]{Proposition}
\newtheorem{lem}[thm]{Lemma}
\newtheorem{result}[thm]{Result}
\newtheorem{cor}[thm]{Corollary} 

\theoremstyle{definition}
\newtheorem{dfn}[thm]{Definition}
\newtheorem{rem}[thm]{Remark}

\allowdisplaybreaks

\title[Circular law for random block band matrices]{Circular law for random block band matrices with genuinely sublinear bandwidth}
\author{Vishesh Jain}
\address{Department of Statistics, Stanford University}
\author{Indrajit Jana}
\address{Indian Institute of Technology, Bhubaneswar}
\author{Kyle Luh}
\address{	Department of Mathematics, 
		University of Colorado Boulder}
\author{Sean O'Rourke}
\address{Department of Mathematics,
	University of Colorado Boulder}
\thanks{Corresponding author email: Sean.D.ORourke@colorado.edu}

\date{\today}

\begin{document}

	\begin{abstract} 
		We prove the circular law for a class of non-Hermitian random block band matrices with genuinely sublinear bandwidth.  Namely, we show there exists $\tau \in (0,1)$ so that if the bandwidth of the matrix $X$ is at least $n^{1-\tau}$ and the nonzero entries are iid random variables with mean zero and slightly more than four finite moments, then the limiting empirical eigenvalue distribution of $X$, when properly normalized, converges in probability to the uniform distribution on the unit disk in the complex plane.  The key technical result is a least singular value bound for shifted random band block matrices with genuinely sublinear bandwidth, which improves on a result of [N. Cook, Ann. Probab., 46, 3442 (2018)] in the band matrix setting.
	\end{abstract}
	
	\maketitle 
	
	\section{Introduction}

	Random band matrices play an important role in mathematics and physics.  
	Unlike many classical matrix ensembles, band matrices with small bandwidth are not of mean-field type and involve short-range interactions.  
	As such, band matrices interpolate between classical mean field models with delocalized eigenvectors (when the bandwidth is large) and models with localized eigenvectors and poisson eigenvalue statistics (when the bandwidth is small) \cite{MR3966510}.  
	In addition, random band matrices have been studied in the context of nuclear physics, quantum chaos, theoretical ecology, systems of interacting particles, and neuroscience \cite{casati1990scaling,wigner1955,Imry_1995,PhysRevLett.75.3501,MR3403052,PhysRevLett.114.088101,Allesina2015PredictingTS,https://doi.org/10.1007/s10144-014-0471-0}.  Many mathematical results have been established for the eigenvalues and eigenvectors of random band matrices, especially Hermitian models; we refer the reader to \cite{anderson2006clt,erdos2013averaging,li2013central,jana2019clt,erdHos2013delocalization,jana2017distribution,schenker2009eigenvector,olver2017phasetransition,JanSahSos15,liu_wang2011,molchanov1992limiting,bb2011,shcherbina2015,khorunzhy2004spectral,bogachev1991level,erdhos2011quantum,bourgade2019random,yang2018random,dimitri_1996,casati1990scaling,fyodorov1991scaling,casati1991scaling,bandeira2016sharp,sodin2010spectral,mirlin1996transition,casati1993wigner,dubach2019words,MR2511659,MR3403052} and references therein.  
	
	In this paper, we focus on non-Hermitian random block band matrices.  Before we introduce the model, we define some notation and recall some previous results for non-Hermitian random matrices with independent entries.  
	For an $n \times n$ matrix $A$, we let $\lambda_1(A), \ldots, \lambda_n(A) \in \mathbb{C}$ denote the eigenvalues of $A$ (counted with algebraic multiplicity).  $\mu_A$ is the empirical spectral measure of $A$ defined as 
	\[ \mu_A := \frac{1}{n} \sum_{i=1}^n \delta_{\lambda_i(A)}, \]
	where $\delta_z$ denotes a point mass at $z$.
	
	The circular law describes the limiting empirical spectral measure for a class of random matrices with independent and identically distributed (iid) entries.  
	\begin{dfn}[iid matrix]
		Let $\xi$ be a complex-valued random variable.  An $n \times n$ matrix $X$ is called an \emph{iid random matrix with atom variable (or atom distribution) $\xi$} if the entries of $X$ are iid copies of $\xi$. 
	\end{dfn}
	
	The circular law asserts that if $X$ is an $n \times n$ iid random matrix with atom variable $\xi$ having mean zero and unit variance, then the empirical spectral measure of $X/\sqrt{n}$ converges almost surely to the uniform probability measure on the unit disk centered at the origin in the complex plane.  This was proved by Tao and Vu in \cite{tao2008random,tao2010random}, and is the culmination of a large number of results by many authors \cite{girko1984circular,MR1310560,bai1997circular,gotze2010circular,edelman1997probability, ginibre1965statistical, MR2129906,mehta1967random}.  We refer the reader to the survey \cite{bordenave2012around} for more complete bibliographic and historical details.  Local versions of the circular law have also been established \cite{MR3230002,MR3230004,MR3278919,MR3683369,MR3770875}.  The eigenvalues of other models of non-Hermitian random matrices have been studied in recent years; see, for instance, \cite{litvak2018circular,cook2017circular,rudelson2018sparse,wood2012universality,girko1986elliptic,naumov2012elliptic,gotze2014one,o2015products,MR2861673,nguyen2014elliptic,cook2018non,2007.15438,MR3481269,MR3359233,MR2892961,MR2857250,MR2837123,MR3798243,MR3189068,MR4029149} and references therein.  
	
	Another model of non-Hermitian random matrices takes the form $X \odot A$, where the entries of the $n \times n$ matrix $X$ are iid random variables with mean zero and unit variance and $A$ is a deterministic matrix.  Here, $A \odot B$ denotes the Hadamard product of the matrices $A$ and $B$, with elements given by $(A \odot B)_{ij} = A_{ij} B_{ij}$.  The matrix $A$ provides the variance profile for the model, and this model includes band matrices when $A$ has a band structure.  The empirical eigenvalue distribution of such matrices was studied in \cite{cook2018non}.  For example, the following result from \cite{cook2018non} describes sufficient conditions for the limiting empirical spectral distribution to be given by the circular law.  
	
\begin{thm}[Theorem 2.4 from \cite{cook2018non}]
Let $\xi$ be a complex-valued random variable with mean zero, unit variance, and $\E|\xi|^{4 + \epsilon} < \infty$ for some $\epsilon > 0$.  Let $X$ be an $n \times n$ iid matrix with atom variable $\xi$, and let $A = (\sigma_{ij}^{(n)})$ be an $n \times n$ matrix with non-negative entries which satisfy 
\begin{equation} \label{eq:varbnd}
	\sup_{n \geq 1} \max_{1 \leq i, j \leq n} \sigma_{ij}^{(n)} \leq \sigma_{\max} 
\end{equation} 
for some $\sigma_{\max} \in (0, \infty)$ and 
\begin{equation} \label{eq:varprofile}
	\frac{1}{n} \sum_{i=1}^n (\sigma_{ij}^{(n)})^2 = \frac{1}{n} \sum_{j=1}^n (\sigma_{ij}^{(n)})^2 = 1 
\end{equation} 
for all $1 \leq i,j \leq n$.  Then the empirical spectral measure of $\frac{1}{\sqrt{n}} X \odot A$ converges in probability as $n \to \infty$ to the uniform probability measure on the unit disk in the complex plane centered at the origin.   
\end{thm}

More generally, the results in \cite{cook2018non} also apply to cases when conditions \eqref{eq:varbnd} and \eqref{eq:varprofile} are relaxed and the limiting empirical spectral measure is not given by the circular law.  
However, the results in \cite{cook2018non}, unlike the results in this paper, require the number of non-zero entries to be proportional to $n^2$ for the limit to be non-trivial.  
	
	\subsection{The model and result}
	In this paper, we focus on a model where the number of non-zero entries is \emph{polynomially} smaller than $n^2$.  
	We now introduce the model of random block band matrices we will study.  
	\begin{dfn}[Periodic block band matrix]\label{defn: block-band}
		Let $b_n \geq 1$ be an integer that divides $n$, and let $\xi$ be a complex-valued random variable.  We consider the \emph{$n \times n$ periodic block-band matrix $\tilde X$ with atom variable (or atom distribution) $\xi$ and bandwidth $b_n$} defined to be the tri-diagonal periodic block band matrix $\tilde X$ given by
		\begin{align}\label{block_band_model}
		\tilde X := \left( \begin{array}{cccccc} 
		\tilde D_1 & \tilde U_2 & & & \tilde T_m  \\
		\tilde T_1 & \tilde D_2 & \tilde U_3 & &  \\
		& \tilde T_2 & \tilde D_3 & \ddots & \\
		& & \ddots & \ddots & \tilde U_{m}  \\
		\tilde U_1 & & &  \tilde T_{m-1} & \tilde D_m
		\end{array}\right)
		\end{align}
		where the entries not displayed are taken to be zero.  Here, $\tilde D_1, \tilde U_1, \tilde T_1, \ldots, \tilde D_m, \tilde U_m, \tilde T_m$ are $b_{n} \times b_n$ independent iid random matrices each having atom variable $\xi$ and $m := n/b_{n}$. For convenience, we use the convention that the indices wrap around; meaning for example that $\tilde U_{-1} = \tilde U_m$. 
	\end{dfn}
	
	Note that each row and column of $\tilde X$ has $3b_{n}$ many nonzero random variables. Using the notation $[m] := \{1, \ldots, m\}$ for the discrete interval, we define
	\begin{align}
	c_{n}&:=3b_{n}&D_{i}&:=\frac{1}{\sqrt{c_{n}}}\tilde{D}_{i},\;\;\forall\;i\in [m]\nonumber\\
	U_{i}&:=\frac{1}{\sqrt{c_{n}}}\tilde{U}_{i},\;\;\forall\;i\in [m]&T_{i}&:=\frac{1}{\sqrt{c_{n}}}\tilde{T}_{i},\;\;\forall\;i\in [m]\nonumber\\
	X&:=\frac{1}{\sqrt{c_{n}}}\tilde{X}&%X_{z}&:=X-zI.
	\label{defn of X_z}
	\end{align}
	
	One motivation for the periodic block band matrix introduced above comes from theoretical ecology. Population densities and food webs, for example, can be modeled by a system involving a large random matrix \cite{https://doi.org/10.1007/s10144-014-0471-0, May1972WillAL}.  The eigenvalues of this random matrix play an important role in the analysis of the the stability of the system, and the circular law and elliptic law have previously been exploited for this purpose \cite{https://doi.org/10.1007/s10144-014-0471-0}.   It has been observed that many of these systems correspond to sparse random matrices with block structures (known as ``modules'' or ``compartments'') \cite{https://doi.org/10.1007/s10144-014-0471-0, Stouffer3648}.  The periodic block band matrix introduced above is one such model with a very specific network structure.

	Our main result below establishes the circular law for the periodic block band model defined above when $b_n$ is genuinely sublinear. To the best of our knowledge, this is the first result to establish the circular law as the limiting spectral distribution for matrices with genuinely sublinear bandwidth.  
	
	\begin{thm}[Circular law for random block band matrices] \label{thm: main theorem}
		There exists $c > 0$ such that the following holds.  Let $\xi$ be a complex-valued random variable with mean zero, unit variance, and $\E |\xi|^{4 + \epsilon} < \infty$ for some $\epsilon > 0$.  Assume $\tilde X$ is an $n \times n$ periodic block-band matrix with atom variable $\xi$ and bandwidth $b_n$, where $cn \geq b_n \geq n^{32/33} \log n$.  Then the empirical spectral measure of $X := \tilde{X} / \sqrt{3b_n}$ converges in probability as $n \to \infty$ to the uniform probability measure on the unit disk in the complex plane centered at the origin.  
	\end{thm}
%	\begin{rem}
%		The proof reveals that $\tau$ can be taken to be $\tau := 1/33$, although this particular value can likely be improved by optimizing some of the exponents in the proof.  
%	\end{rem}

	We prove Theorem \ref{thm: main theorem} by showing that there exists constants $c, \tau > 0$ so that the empirical spectral measure of $X$ converges to the circular law under the assumption that the bandwidth $b_n$ satisfies $cn \geq b_n \geq n^{1-\tau} \log n$.  In fact, the proof reveals that $\tau$ can be taken to be $\tau := 1/33$, as stated in Theorem \ref{thm: main theorem}, although this particular value can likely be improved by optimizing some of the exponents in the proof.  
	
	A few remarks concerning the assumptions of Theorem \ref{thm: main theorem} are in order.  First, the restriction on the bandwidth $b_n \geq n^{1- \tau} \log(n)$ with $\tau = 1/33$ is of a technical nature and we believe this condition can be significantly relaxed.  For instance, we give an exponential lower bound on the least singular value of $X - zI$ for $z \in \mathbb{C}$ in Theorem \ref{thm:lsv} below.  If this bound could be improved to say polynomial in $n$, then we could improve the value of $\tau$ to $1/2$.  It is possible that other methods could also improve this restriction even further.  Second, the assumption that the entries have finite $4+\epsilon$ moments is due to the sublinear bandwidth growth rate. Our calculation requires higher moment assumptions for slower bandwidth growth, as can be seen from the proof of Theorem \ref{thm: JanaSos2017 theorem}. 
	
	A numerical simulation of Theorem \ref{thm: main theorem} is presented in Figure \ref{fig:numerical}.  
	
	\begin{figure}[h!]
		\centering
		\begin{subfigure}[b]{0.4\linewidth}
			\includegraphics[trim=100 250 100 250,clip,width=\linewidth]{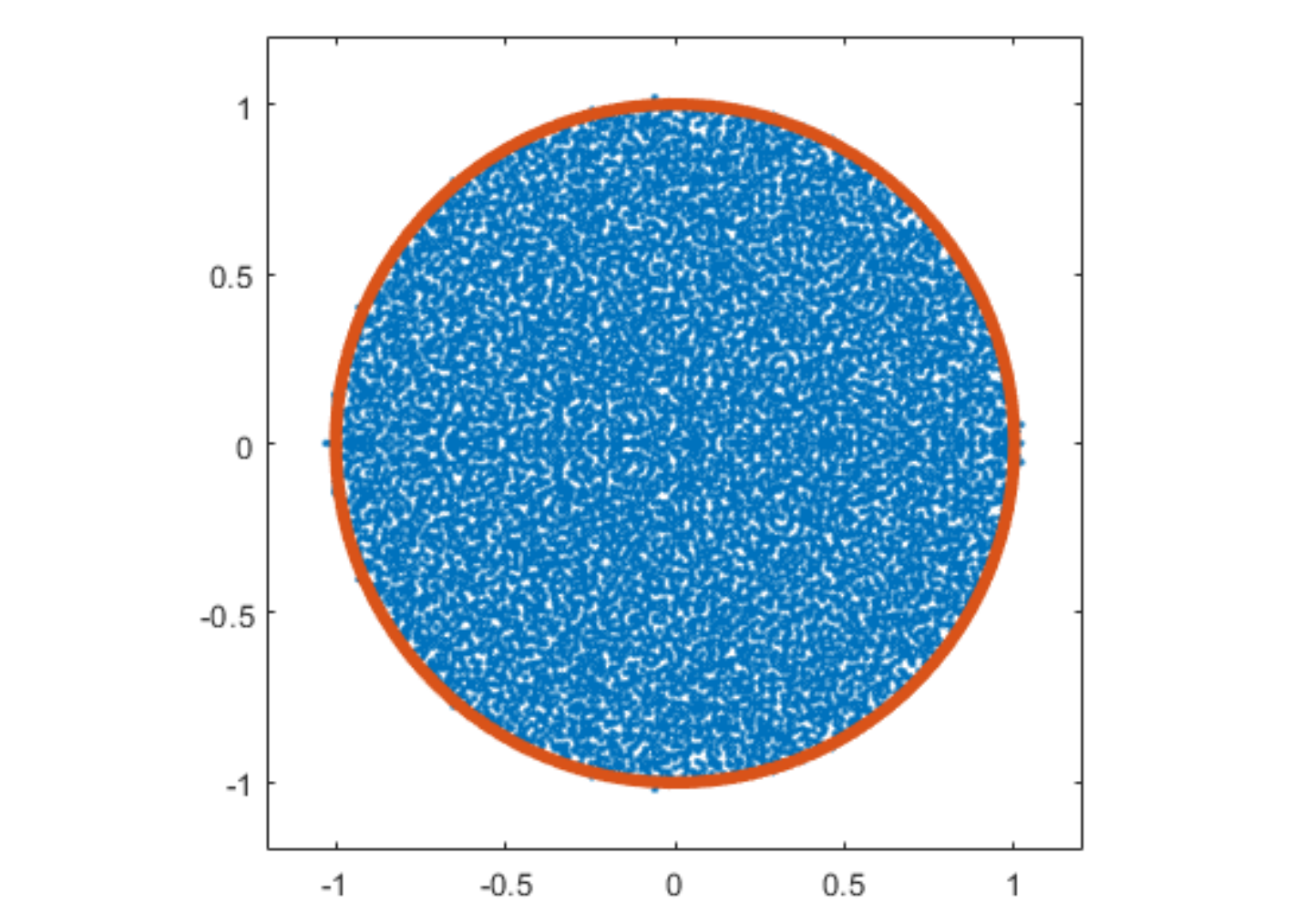}
			\caption{$\tilde{X}$ has Gaussian atom variable with $n=10,000$ and $b_n = 100$.}
		\end{subfigure}
		\begin{subfigure}[b]{0.4\linewidth}
			\includegraphics[trim=100 250 100 250,clip,width=\linewidth]{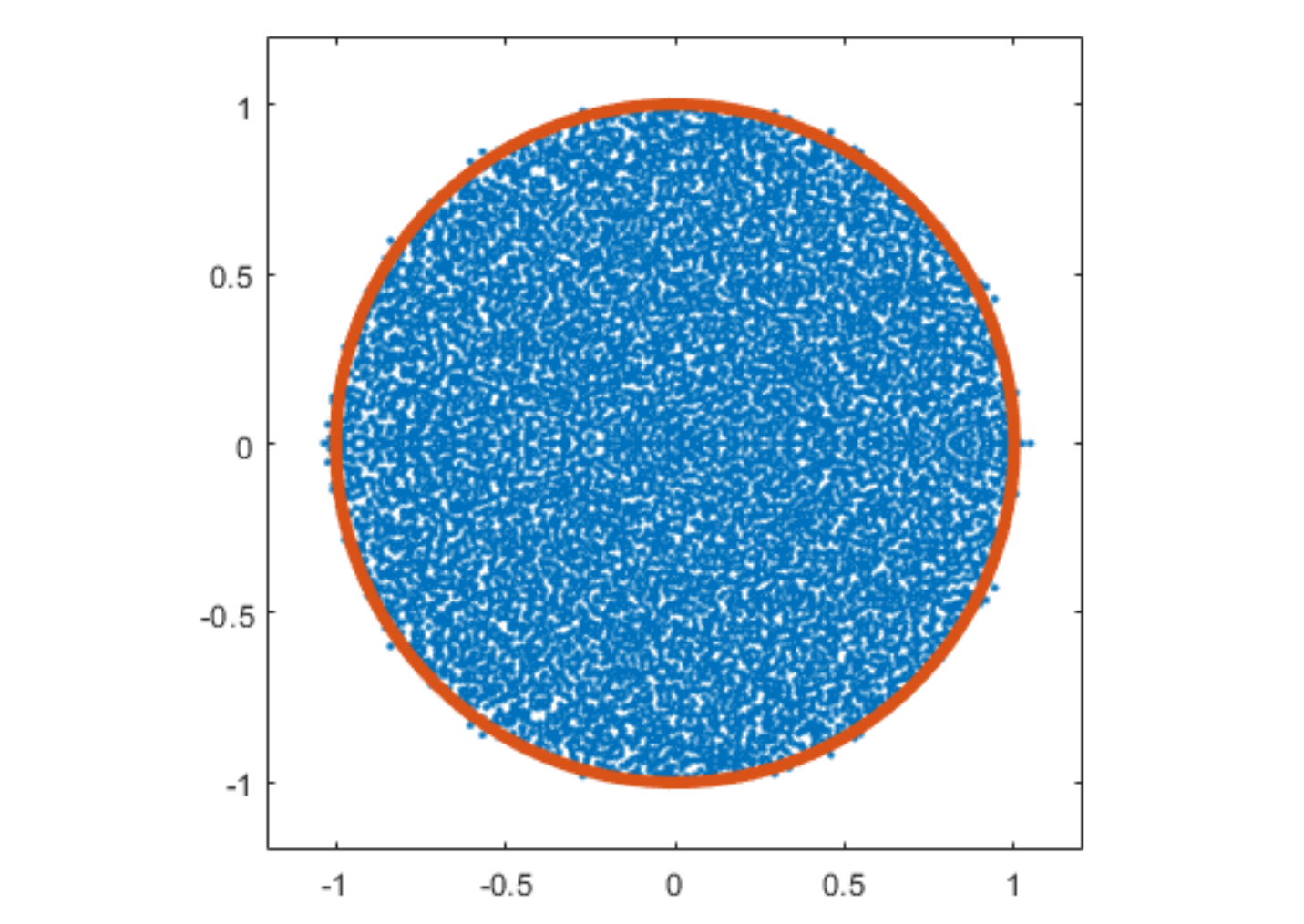}
			\caption{$\tilde{X}$ has Rademacher atom variable with $n=10,000$ and $b_n = 100$.}
		\end{subfigure}
		\begin{subfigure}[b]{0.4\linewidth}
			\includegraphics[trim=100 250 100 250,clip,width=\linewidth]{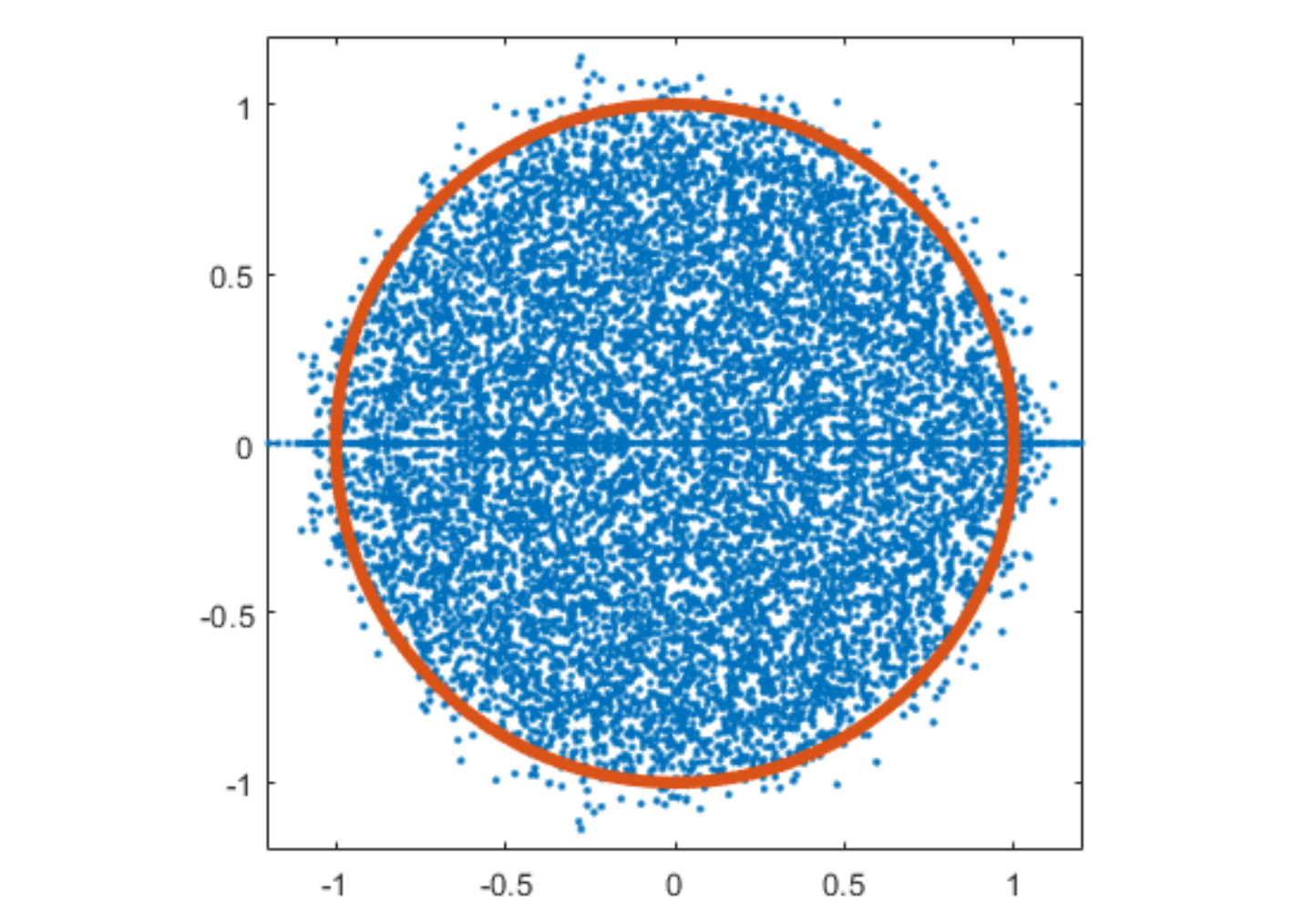}
			\caption{$\tilde{X}$ has Gaussian atom variable with $n=10,000$ and $b_n = 10$.}
		\end{subfigure}
		\begin{subfigure}[b]{0.4\linewidth}
			\includegraphics[trim=100 250 100 250,clip,width=\linewidth]{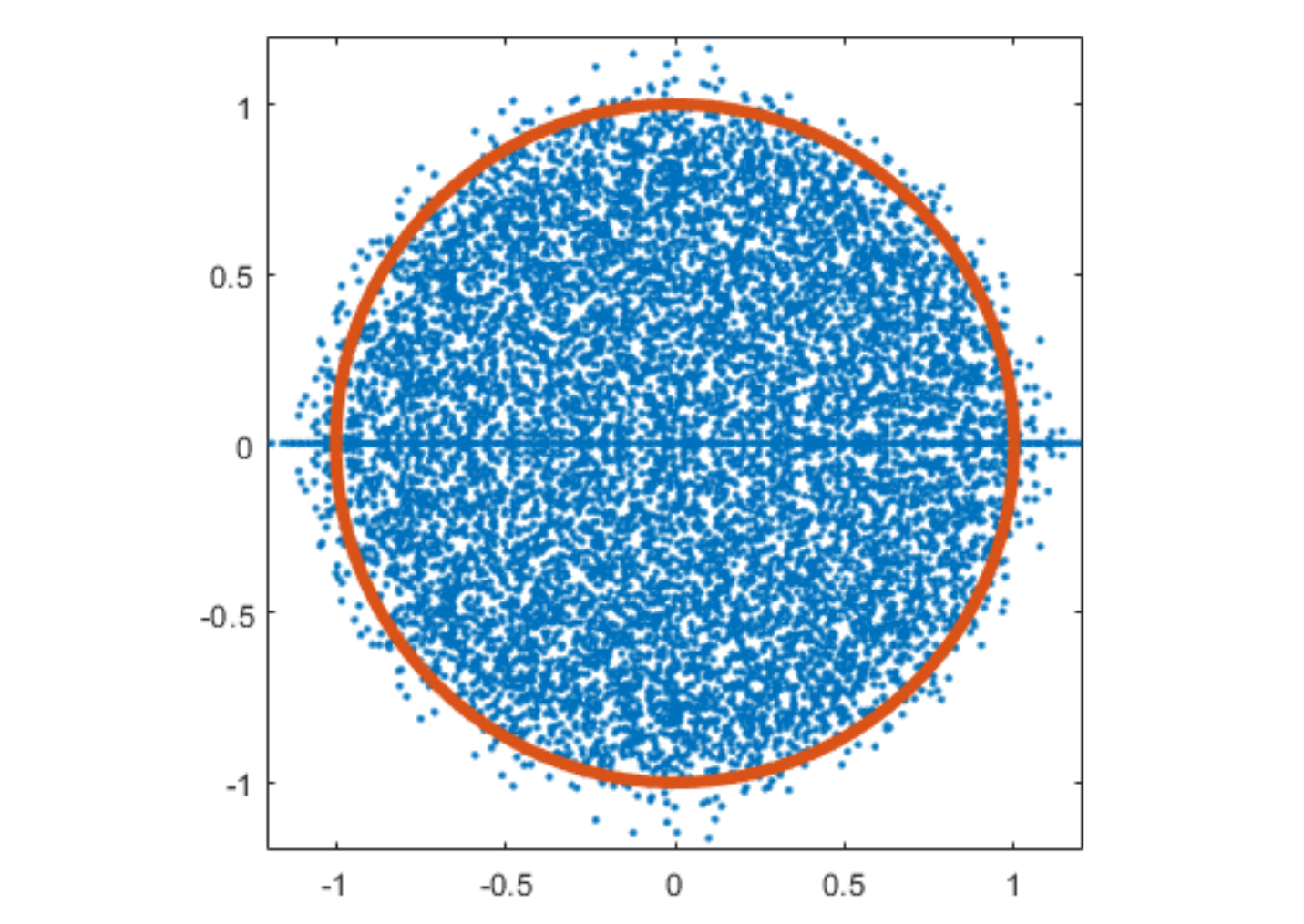}
			\caption{$\tilde{X}$ has Rademacher atom variable with $n=10,000$ and $b_n = 10$.}
		\end{subfigure}
		\caption{Numerical simulations for the eigenvalues of $X := \tilde{X} / \sqrt{3 b_n}$ when $\tilde{X}$ is an  $n \times n$ period block-band matrix with bandwidth $b_n$ for various atom distributions.}
		\label{fig:numerical}
	\end{figure}

	\subsection{Notation and overview}\label{sec: Preliminaries} 
	We use asymptotic notation under the assumption that $n \to \infty$.  The notations $X = O(Y)$ and $Y= \Omega(X)$ denote the estimate $|X| \leq C Y$ for some constant $C >0$ and all $n \geq C$.  We write $X = o(Y)$ if $|X| \leq c_n Y$ for some $c_n$ that goes to zero as $n$ tends to infinity.  
	
	For convenience, we do not always indicate the size of a matrix in our notation. For example, to denote an $n\times n$ matrix $A$, we simply write $A$ instead of $A_{n}$ when the size is clear. We use $b_{n}$ to denote the size of each block matrix and $c_{n}:=3b_{n}$ for the number of non-zero entries per row and column.  We let $[n]:=\{1,2,3,\ldots,n\}$ and $e_{1},e_{2}, \ldots,e_{n}$ be the standard basis elements of $\C^{n}$.
	For a matrix $A$, $a_{ij}$ will be the $(i,j)$-th entry, $a_{k}$ will be the $k$th column, $A^{(k)}$ represents the matrix $A$ with its $k$th column set to zero and $\CH_{k}$ will be the span of the columns of $A^{(k)}$.
	Furthermore, $A^{*}$ is the complex conjugate transpose of the matrix $A$, and when $A$ is a square matrix, we let
	\[A_{z}:=A-zI\]
	where $I$ denotes the identity matrix and $z \in \mathbb{C}$.  
	
	For the spectral information of an $n \times n$ matrix $A$, we designate 
	\[\lambda_{1}(A),\lambda_{2}(A),\ldots, \lambda_{n}(A) \in \mathbb{C} \] 
	to be the eigenvalues of $A$ (counted with algebraic multiplicity) and
	\[\mu_{A}:=\frac{1}{n}\sum_{i=1}^{n}\delta_{\lambda_{i}(A)}\]
	to be the empirical measure of the eigenvalues.  Here, $\delta_z$ represents a point mass at $z \in \mathbb{C}$.  
	Similarly, we denote the singular values of $A$ by
	\[s_{1}(A)\geq s_{2}(A)\geq \ldots\geq s_{n}(A)\geq 0 \]
	and the empirical measure of the squared-singular values as
	\[\nu_{A}:=\frac{1}{n}\sum_{i=1}^{n}\delta_{s_{i}^{2}(A)}.\]
	Additionally, we use $\|A\|$ to mean the standard $\ell_2 \to \ell_2$ operator norm of $A$.  
	
	For a vector $ v \in \C^n$,
	\[\|v\|:=\left(\sum_{k=1}^{n}|v_{k}|^{2}\right)^{1/2}\; \text{and}\; \|v\|_{\infty}:=\max_{k}|v_{k}|.\]
	
	Finally, we use the following standard notation from analysis and linear algebra.  
	The set of unit vectors in $\mathbb{C}^{n}$ will be denoted by $S^{n-1}$ i.e. $S^{n-1}:=\{v\in \C^{n}:\|v\|=1\}$ and the disk of radius $r$ by $\D_{r}:=\{z\in \C:|z|<r\}$.  For any set $\CS \subset \C^n$ and $u \in \mathbb{C}^n$,
	\[\dist(u,\CS):=\inf_{v\in \CS}\|u-v\|.\]
	$|S|$ denotes the cardinality of the finite set $S$.

	The rest of the paper is devoted to the proof of Theorem \ref{thm: main theorem}.  The proof proceeds via Girko's Hermitization procedure (see \cite{bordenave2012around}) which is now a standard technique in the study of non-Hermitian random matrices. Following \cite{jana2017distribution}, we study the empirical eigenvalue distribution of $X_z X_z^\ast$ for $z \in \mathbb{C}$.  In particular, we establish a rate of convergence for the Stieltjes transform $X_z X_z^\ast$ to the Stieltjes transform of the limiting measure in Section \ref{sec:esd}.  The key technical tool in our proof is a lower bound on the least singular value of $X_z$ presented in Section \ref{sec:lsv}.  In Section \ref{sec: proof of the main theorem}, following the method of Bai \cite{bai1997circular}, these two key ingredients are combined and the proof of Theorem \ref{thm: main theorem} is given.  The appendix contains a number of auxiliary results.

	\section{Least singular value} \label{sec:lsv}
	
	In this section, we present our key least singular value bound, Theorem~\ref{thm:lsv}. The crucial feature of our result is that the lower bound on the least singular value is only singly exponentially small in $m$. While this is most likely suboptimal, and indeed, we conjecture that our bound can be substantially improved, it is still significantly better than previous results in the literature. Notably, the work of Cook \cite{cook2018lower} provides lower bounds on the least singular value for more general structured sparse random matrices; however, specialized to our setting, the lower bound there is \emph{doubly} exponentially small in $m$ (see Equation~3.8 in \cite{cook2018lower}), which only translates to a circular law for bandwidth (at best) $\Omega(n/\log{n})$.\\    
	
	We consider the translated periodic block-band model $X_{z} = X - zI$, where $X$ is as defined in \eqref{defn of X_z} and $z \in \mathbb{C}$ is fixed. Recall that $m = n/b_n$.  Throughout this section, we will assume that $b_{n}\geq m \geq m_0$, where $m_0$ is a sufficiently large constant. Recall that for an $n\times n$ matrix $A$, we let $s_1(A) \geq s_2(A) \geq \dots \geq s_n(A)\geq 0$ denote its singular values. 
	
	%In this section, we establish the following lower tail estimate on the least singular value of the translated periodic block band matrix $X_{z}$ (as defined in \eqref{defn of X_z}) for a fixed $z \in \mathbb{C}$.
	%\todo{Maybe $\D_{C}$?}
	\begin{thm}\label{thm:lsv} 
		Fix $\epsilon, K' > 0$. Suppose $\tilde X$ is an $n \times n$ periodic block band matrix \emph{(}as defined in \eqref{block_band_model}\emph{)} with atom variable $\xi$ satisfying $\E[\xi] = 0$, $\E[|\xi|^{2}]=1$, and $\E[|\xi|^{4+\epsilon}] \leq C$, for some absolute constant $C > 0$. Then, for any $z \in \mathbb{C}$ such that $|z| \leq K'$,
		$$
		\Pb(s_n(X_{z}) \leq c_{n}^{-25m}) \leq \frac{C_{\xi}}{\sqrt{c_{n}}},
		$$	
		where $C_{\xi}$ is a constant depending only on $\epsilon$, $C$ and $K'$.  
	\end{thm}
	
	Let us define the event
	\begin{align*}
	\CE_{K}=\left\{\forall i\in [m]:\; \; \|U_{i}\|,\|(D_{i})_{z}\|,\|T_{i}\|\leq K, \text{ and }\;s_{b_{n}}(U_{i}),s_{b_{n}}(T_{i})\geq b_{n}^{-5}\right\}.
	\end{align*}
	We begin by showing (Lemma~\ref{lem:EKc}) that $\Pb(\CE_{K}^{c})=O(1/c_n)$.
	%Let $\CE_K$ denote the event that for all $i \in [m]$, $\|U_i\|, \|D_i\|, \|T_i\| \leq K\sqrt{b_{n}}$, and  $s_{b_{n}}(U_i), s_{b_{n}}(T_i) \geq b_{n}^{-5}$. We begin by showing that this event holds with sufficiently high probability. 
	This will allow us to restrict ourselves to the event $\CE_K$ for the remainder of this section. 
	
	In order to bound the probability of the event $\CE_{K}^{c}$, we will need the following two results on the smallest and largest singular values of (shifts of) complex random matrices with iid entries. 
	
	\begin{prop}[Theorem 1.1 from \cite{jain2019strong}]
		Let $A$ be an $n \times n$ matrix whose entries are iid copies of a complex random variable $\xi$ satisfying $\E[\xi] =0$ and $\E[|\xi|^{2}]=1$. Let $F$ be a fixed $n \times n$ complex matrix whose operator norm is at most $n^{0.51}$. Then, for any $\varepsilon \geq 0$, 
		$$
		\Pb(s_{n}(F+A) \leq \varepsilon n^{-5/2}) \leq C \varepsilon + C \exp(-\gamma n^{1/50})
		$$   
		for two constants $C > 0, \gamma \in (0, 1)$ depending only on the distribution of the random variable $\xi$. 
	\end{prop}
	
	The next proposition can be readily deduced from Theorem 5.9 in \cite{bai2010spectral} along with the standard Chernoff bound. 
	\begin{prop}\label{prop:largestsing}
		Fix $\epsilon > 0$. Let $A$ be an $n \times n$ matrix whose entries are iid copies of a complex random variable $\xi$ satisfying $\E[\xi] =0$, $E[|\xi|^{2}]=1$ and $\E[|\xi|^{4+\epsilon}]\leq M$. Then
		$$\mathbb{P}[\|A\| > K\sqrt{n}] \leq Kn^{-2},$$
		where $K > 0$ is a sufficiently large constant depending only on $\xi$ \emph{(}and hence, also the parameter $\epsilon > 0$\emph{)}.
	\end{prop}

	Applying the above two propositions (along with the triangle inequality for $\|(D_i)_z\|$) and using the union bound, we immediately obtain: 
	\begin{lem} \label{lem:EKc}
		There exists a constant $K > 0$, depending only on $|z|$ and the random variable $\xi$ \emph{(}and hence also on the parameter $\epsilon > 0$\emph{)} such that
		$$
		\Pb(\CE_K^c) \leq Kb_{n}^{-1}.
		$$
	\end{lem}

For the remainder of this section, we will restrict ourselves to the event $\CE_K$. For any $v \in \C^n$, we let
	$$
	v = \left(\begin{array}{c}
	v_{[1]} \\
	v_{[2]} \\
	\vdots \\
	v_{[m]}
	\end{array} \right)
	$$  
	be the division of the coordinates into $m$ vectors $v_{[i]} \in \C^{b_{n}}$. We will use  $v_i$ to denote the $i$-th coordinate of $v$. 
	For convenience, we use the convention that the indices wrap around meaning, for example, that $v_{[m+1]} = v_{[1]}$.
	
	%Since restricted to the event $\CE_K$, $\|U_{i}\|, \|D_{i}\|, \|T_{i}\| \leq K\sqrt{b}$ for all $i\in [m]$, it is immediate that restricted to $\CE_K$, $\|X - z \sqrt{b}\| \leq 10 K \sqrt{b}$.  
	For $\alpha, \beta \in (0,1)$, let
	$$
	L_{\alpha,\beta} := \{v \in S^{n-1}: |\{i \in [n]: |v_i| \geq \beta b_{n}^{-10 m} n^{-1/2} \}| \geq \alpha n \},
	$$
	i.e. $L_{\alpha,\beta}$ consists of those unit vectors that have sufficiently many large coordinates. For us, $\alpha$ and $\beta$ are constants depending on $K$ which will be specified later. 
	Then, as $s_n(X_{z}) = \inf_{v \in S^{n-1}} \|X_{z}v\|$, we can decompose the least singular value problem into two terms:
	\begin{align} \label{eq:decomp}
	&\Pb(\CE_K \cap \{s_n(X_{z})\leq t b_{n}^{-10 m} n^{-1/2}\})\\
	\leq &\Pb(\CE_K \cap \{\inf_{v \in L_{\alpha,\beta}} \|X_{z} v\|\leq tb_{n}^{-10 m} n^{-1/2}\}) + \Pb(\CE_K \cap \{\inf_{v \in L_{\alpha,\beta}^c} \| X_{z}v\| \leq tb_{n}^{-10 m} n^{-1/2}\}). \nonumber
	\end{align}
	
	\subsection{Reduction to the Distance Problem}
	We begin with a lemma due to Rudelson and Vershynin, which converts the first term in (\ref{eq:decomp}) into a question about the distance of a random vector to a random subspace. 
	\begin{lem}[Lemma 3.5 from \cite{rv2008invertibility}] \label{lem:distance}
		Let $x_1-ze_{1}, \dots, x_n-ze_{n}$ be the columns of $X_{z}$ and let $\CH_{i}$ be the span of all the columns except the $i$-th. Then,  
		$$
		\Pb(\CE_K \cap \{\inf_{v \in L_{\alpha,\beta}} \|X_{z}v\| \leq t b_{n}^{-10 m} n^{-1/2}\}) \leq \frac{1}{\alpha n} \sum_{k=1}^n \Pb(\CE_K \cap \{\dist(x_k-ze_{k}, \CH_{k}) \leq \beta^{-1}t\}).
		$$
	\end{lem}
	\begin{proof}
		Let 
		$$
		p_{k} :=  \Pb(\CE_K \cap \{\text{dist}(x_k-ze_{k}, \CH_{k}) \leq \beta^{-1}t\}).  
		$$  By the linearity of expectation, we have 
		$$
		\E|\{k \in [n]:\CE_K \text{ and } \{\text{dist}(x_k-ze_{k}, \CH_{k}) \leq \beta^{-1}t\} \}| = \sum_{k=1}^{n}p_{k}.
		$$
		Therefore, if we let 
		\[\Xi = \CE_K \cap \{|\{k \in [n]: \text{dist}(x_k-ze_{k}, \CH_{k}) \leq \beta^{-1}t \}| < \alpha n\},\] it follows from Markov's inequality that
		\begin{align*}
		\Pb(\CE_K \cap \Xi^c) 
		\leq \frac{\sum_{k=1}^{n}p_{k}}{\alpha n}.
		\end{align*}
		By definition, any vector $v \in L_{\alpha,\beta}$ has at least $\alpha n$ coordinates with absolute value larger than $\beta b_{n}^{-10 m} m^{-1/2}b_{n}^{-1/2}$.  Therefore, on the event $\Xi$, for any $v \in L_{\alpha,\beta}$, there exists some $k \in [n]$ such that $|v_k| \geq \beta b_{n}^{-10m}m^{-1/2}b_{n}^{-1/2}$ and $\dist(x_k-ze_{k},\CH_{k}) > \beta^{-1}t$. Hence,  on the event $\Xi$, for all $v \in L_{\alpha,\beta}$,
		\begin{align*}
		\|X_{z} v\| \geq |v_k| \text{dist}(x_k-ze_{k}, \CH_{k}) \geq t b_{n}^{-10 m} m^{-1/2}b_{n}^{-1/2}.
		\end{align*}
		Thus, we see that the probability of the event in the statement of the lemma is at most the probability of $\CE_K \cap \Xi^{c}$, which gives the desired conclusion. 
		\end{proof}
	
	The distance of $x_k-ze_{k}$ to $\CH_{k}$ can be bounded from below by $|\langle x_k-ze_{k}, \hat n\rangle|$ where $\hat n$ is a unit vector orthogonal to $\CH_{k}$.  Our next goal is to obtain some structural information about any vector normal to $\CH_{k}$. For  convenience of notation, we will henceforth assume that $k=1$; the same arguments are readily seen to hold for other values of $k$ as well. Moreover, since the distribution of $X_{z}$ is invariant under transposition, we may as well assume that $x_{1}-ze_{1}$ is the first \emph{row} of $X_{z}$ and that $\CH_{1}$ is the subspace spanned by all the rows except for the first.  
	
	\subsection{Structure of Normal Vectors and Approximately Null Vectors}
	Recall that $\CH_{1}$ is the subspace generated by all the rows of $X_{z}$ except for the first row. The next proposition establishes that if $v$ is normal to $\CH_{1}$, then there are sufficiently many $v_{[i]}$ with large enough  norm. Our approach to lower bounding the coordinates of $v$ is similar to the methods used in \cite{MR3695802}; our proof is also similar in spirit to the proof of Proposition 2.9 in \cite{cjp2020}.   
	\begin{prop} \label{prop:consecutivenorm}
		On the event $\CE_K$, for any vector $v \in S^{n-1}$ that is orthogonal to $\CH_{1}$ and for all sufficiently large $n$ \emph{(}depending on $K$\emph{)}, either 
		$$
		\|v_{[i]}\| \geq b_{n}^{-10 m} m^{-1/2} \text{ or } \|v_{[i+1]}\| \geq b_{n}^{-10 m} m^{-1/2}.
		$$
		for all $i \in [m-1]$. 
	\end{prop} 
	\begin{proof}
		By definition, $v$ must satisfy the following collection of equations:
		\begin{align} \label{eq:normalvector}
		\nonumber T_1 v_{[1]} + (D_2)_{z}v_{[2]} + U_3 v_{[3]} &= 0 \\
		\nonumber &\, \, \, \vdots \\ 
		T_{i-1} v_{[i-1]} + (D_i)_{z}v_{[i]} + U_{i+1} v_{[i+1]} &= 0 \\
		\nonumber &\, \, \, \vdots \\
		\nonumber T_{m-2} v_{[m-2]} + (D_{m-1})_{z}v_{[m-1]} + U_m v_{[m]}&= 0 \\
		\nonumber T_{m-1} v_{[m-1]} + (D_m)_{z}v_{[m]} + U_1 v_{[1]} &= 0
		\end{align}
		Moreover, since $v \in S^{n-1}$, there exists a smallest index $j_0 \in [m]$ such that $\|v_{[j]}\| \geq m^{-1/2}$.  If $j_0 \geq 3$, then the following equation (which is a part of \eqref{eq:normalvector})
		$$
		T_{j_0 - 2} v_{[j_0-2]} + (D_{j_0-1})_{z} v_{[j_0-1]} + U_{j_0} v_{[j_0]} = 0
		$$  
		implies that 
		$$
		\|T_{j_0 - 2} v_{[j_0-2]} + (D_{j_0-1})_{z} v_{[j_0-1]}\| = \|U_{j_0} v_{[j_0]} \|. 
		$$
	On the event $\CE_K$, we have from the triangle inequality that
		$$
		\|T_{j_0 - 2} v_{[j_0-2]} + (D_{j_0-1})_{z} v_{[j_0-1]}\| \leq K (\|v_{[j_0-2]}\| + \|v_{[j_0-1]}\|)
		$$
		and
		$$
		\|U_{j_0} v_{[j_0]}\| \geq b_{n}^{-5}\|v_{[j_0]}\| \geq b_{n}^{-5}m^{-1/2}.
		$$
		Therefore, for $n$ sufficiently large compared to $K$, either
		\begin{equation} \label{eq:normbound}
		\|v_{[j_0 -2]}\| \geq b_{n}^{-10} m^{-1/2} \text{ or } \|v_{[j_0 -1]}\| \geq b_{n}^{-10} m^{-1/2}.
		\end{equation}
		Now, let $j_{-1}$ be the smaller of the two indices $j_0-1$ and $j_0-2$ that satisfies (\ref{eq:normbound}).  Recall that, for convenience, we are considering indices modulo m.  If $j_{-1} \geq 3$ then iterating the argument with $j_{-1}$ and the equation
		$$
		T_{j_{-1} - 2} v_{[j_{-1}-2]} + (D_{j_{-1}-1})_{z} v_{[j_{-1}-1]} + U_{j_{-1}} v_{[j_{-1}]} = 0,
		$$
		we can find $j_{-2} \in \{j_{-1} -1, j_{-1}-2\}$ such that
		$$
		\|v_{[j_{-2}]}\| \geq b_{n}^{-20} m^{-1/2}.  
		$$
		Continuing in this manner, we will generate a sequence of indices $j_0, j_{-1}, \dots, j_{-k}$, $k\leq m$, such that $j_{-k} \in \{1,2\}$ and such that for all $i\in [k]$
		$$
		|j_{-i} - j_{-i-1}| \leq 2 \text{ and } \|v_{[j_{-i}]}\| \geq b_{n}^{-10  i} m^{-1/2}. 
		$$
		We may apply a similar argument to handle indices larger than $j_0$. Indeed, if $j_0 \leq m-3$, then we have from (\ref{eq:normalvector}) that,
		$$
		T_{j_0} v_{[j_0]} + (D_{j_0+1})_{z} v_{[j_0+1]} + U_{j_0+2} v_{[j_0+2]} = 0. 
		$$
		Once again, on the event $\CE_K$, 
		$$
		\| (D_{j_0+1})_{z} v_{[j_0+1]} + U_{j_0+2} v_{[j_0+2]} \| \leq K(\|v_{[j_0 + 1]} \| + \|v_{[j_0 + 2]}\|)
		$$
		and 
		$$
		\| T_{j_0} v_{[j_0]} \| \geq b_{n}^{-5} m^{-1/2}.
		$$
		As before, this implies that either
		$$
		\|v_{[j_0+1]}\| \geq b_{n}^{-10} m^{-1/2}  \text{ or } \|v_{[j_0+2]}\| \geq b_{n}^{-10} m^{-1/2}.  
		$$
		By iterating this process as above, we obtain a sequence of indices such that $j_0, j_1, \dots, j_{k'}$, $k' \leq m$, such that $j_{k'} \in \{m-1,m\}$ and such that for all $i\in [k']$ 
		$$
		|j_{i} - j_{i-1}| \leq 2 \text{ and } \|v_{[j_{i}]}\| \geq b_{n}^{-10  i} m^{-1/2}.
		$$
		This completes the proof.
	\end{proof}
	
	Note that in the above proof, it is not important that $v$ is precisely normal to $\CH_{1}$. Indeed, exactly the same proof allows us to obtain a similar conclusion for approximately null vectors as well. 
	\begin{prop}\label{prop:approximatenull}
		Restricted to $\CE_K$, for any vector $v \in S^{n-1}$ such that $\|X_{z} v\| \leq b_{n}^{-10 m} m^{-1/2}$ and for all sufficiently large $n$ \emph{(}depending on $K$\emph{)}, either
		$$
		\|v_{[i]}\| \geq b_{n}^{-10 m} m^{-1/2} \text{ or } \|v_{[i+1]}\| \geq b_{n}^{-10 m} m^{-1/2}
		$$
		for all $i \in [m-1]$,
	\end{prop}

	Our next goal is to show that for $\alpha,\beta$ sufficiently small depending on $K$ (indeed, the proof shows that we can take $\alpha < \gamma'/(K^{2}\log{K})$ and $\beta < \gamma'/K$, where $\gamma'>0$ is a constant depending only on the distribution of the random variable $\xi$), we have
	\begin{equation}
	\label{eqn:smallcoord}
	\Pb(\CE_K \cap \{\inf_{v\in L_{\alpha,\beta}^{c} }\|X_{z}v\|\leq \gamma b_{n}^{-10m}m^{-1/2}\}) \leq m\exp(-\gamma b_{n}),
	\end{equation}
	where $\gamma \in (0,1)$ is a constant depending only on the distribution of the random variable $\xi$. \\
	
	For this, we begin with a standard decomposition of the unit sphere, due to Rudelson and Vershynin \cite{rv2008invertibility}.  
	\begin{dfn}
		For $k \in \N$ and $a, \kappa \in (0, 1)$, let $\text{Sparse}_k(a)$ denote the sparse vectors $\{v \in S^{k-1}: |\text{supp}(v)|\leq a k\}$.  We define \emph{compressible} vectors by
		$$
		\text{Comp}_k(a, \kappa) := \{v \in S^{k-1}: \exists u \in \text{Sparse}_k(a) \text{ such that } \|v - u\| \leq \kappa\}.
		$$ 
		and \emph{incompressible} vectors by
		$$
		\text{Incomp}_k(a, \kappa) := S^{k-1} \setminus \text{Comp}_k(a, \kappa).
		$$
	\end{dfn}

	\begin{lem} \label{lem:comp}
		Let $M_i$ denote the $b_{n} \times c_{n}$ block matrix given by $(T_{i-1}  \quad (D_i)_{z} \quad U_{i+1} )$. There exists a constant $\gamma \in (0,1)$, depending only on the distribution of the random variable $\xi$, such that 
		$$
		\Pb\left(\CE_K \cap \{\inf_{w \in \emph{Comp}_{c_{n}}(a, \kappa)} \|M_{i} w\| \leq \gamma\}  \right) \leq \exp(-\gamma b_{n}),
		$$
		where $a = \gamma/\log{K}$ and $\kappa = \gamma/K$.
	\end{lem}
	\begin{proof}
		This is (by now) a standard argument; we include the short proof for the reader's convenience.  We begin with the set $\text{Sparse}_{c_{n}}(a)$.  For any vector $v \in S^{c_{n}-1}$, there exist positive constants $\gamma, \gamma'$, depending only on the distribution of the entries of $M_i$ such that
		$$
		\Pb(\|M_i v\| \leq \gamma ) \leq e^{-\gamma' b_{n}}
		$$
		(cf. Lemma 2.4 in \cite{jain2019strong}).  Recall that an $\varepsilon$-net of a set $U$ is a subset $\mathcal{N} \subset U$ such that for any $w \in U$, there exists a $w' \in \mathcal{N}$ satisfying $\|w - w'\| \leq \varepsilon$.  By a simple volumetric argument, one can construct an $\varepsilon$-net $\mathcal{N}$ of $\text{Sparse}_{c_{n}}(a)$ with 
		$$
		|\mathcal{N}| \leq \binom{c_{n}}{ac_{n}} \left( \frac{3}{\varepsilon} \right)^{ac_{n}} \leq \exp(ac_{n} \log(e/a) + ac_{n} \log(3/\varepsilon)).
		$$
		We set $\varepsilon = \frac{\gamma}{20K}$. Then, by a union bound,
		\begin{align*}
		\Pb(\inf_{v \in \mathcal{N}} \|M_i v\| \leq \gamma ) &\leq \sum_{v \in \mathcal{N}} \Pb(\|M_i v\| \leq \gamma ) \\
		&\leq \exp(ac_{n} \log(e/a) + ac_{n} \log(3/\varepsilon) - \gamma' b_{n}) \\
		&\leq \exp(-\tilde{\gamma} b_{n}),
		\end{align*}
		where the last inequality holds for $a < \gamma''/\log{K}$ (for an absolute constant $\gamma'' > 0$). 
		Let $v \in \text{Sparse}_{c_{n}}(a)$. Then, by definition, there exists some $v' \in \mathcal{N}$ such that $\|v - v'\| \leq \varepsilon$. Therefore, on the event $\inf_{v \in \mathcal{N}} \|M_i v\| > \gamma $, we have for any $v \in \text{Sparse}_{c_{n}}(a)$ that
		$$
		\|M_i v\| \geq \|M_i v\| - \|v - v'\| \|M_i\| \geq \gamma  - \frac{\gamma}{20K} 10 K = \frac{\gamma}{2} . 
		$$
		We can then conclude that 
		$$
		\Pb\left(\inf_{v \in \text{Sparse}_{c_{n}}(a)} \|M_i v\| \leq \frac{\gamma}{2}  \right) \leq \exp(-\tilde{\gamma} b_{n}).
		$$
		To extend this to compressible vectors, we simply choose $\kappa = \frac{\gamma}{40K}$. For any $y \in \text{Comp}_{c_{n}}(a, \kappa)$, there exists $v \in \text{Sparse}_{c_{n}}(a)$ such that $\|y - v\| \leq \kappa$.  Thus, if $\|M_i v\| \geq \gamma /2$ then
		\[
		\|M_i y\| \geq \|M_i v\| - \|M_i\| \|v - y\| \geq \frac{\gamma}{2}  - 10K \frac{\gamma}{40K} \geq \frac{\gamma}{4} . \qedhere 
		\]
	\end{proof}
	
	We will also need the following lemma from \cite{rv2008invertibility}.
	
	\begin{lem}[Lemma 3.4 from \cite{rv2008invertibility}] \label{lem:incompcoordinates}
		If $v \in \text{Incomp}_k(a, \kappa)$, then there exist constants $\gamma_1$ and $\gamma_2$ depending only on $a$ and $\kappa$ such that there are at least $\gamma_1 k$ coordinates with $\gamma_{3}k^{-1/2}\geq |v_i| \geq \gamma_2 k^{-1/2}$. In fact, we can take $\gamma_1 = \kappa^{2} a/2$, $\gamma_2 = \kappa/\sqrt{2}$, and $\gamma_{3} = \kappa^{-1/2}$.  
	\end{lem} 
	
	Now, we are ready to prove \eqref{eqn:smallcoord}.
	Consider a vector $v \in S^{n-1}$ such that $\|X_{z} v\| \leq t b_{n}^{-10m} m^{-1/2}$, where $0\leq t \leq 1$. 
Then, on the event $\CE_K$, it follows from Proposition \ref{prop:approximatenull} that for any $i \in [m]$,
	$$
	\|(v_{[i-1]}, v_{[i]}, v_{[i+1]})\| \geq b_{n}^{-10m} m^{-1/2}.
	$$
	Moreover, since for every $i \in [m]$, 
	$$
	\left\|(U_{i-1}, (D_{i})_{z}, T_{i+1}) \frac{(v_{[i-1]}, v_{[i]}, v_{[i+1]})^T}{\|(v_{[i-1]}, v_{[i]}, v_{[i+1]})^T\|} \right \| \leq \frac{t b_{n}^{-10m} m^{-1/2}}{\|(v_{[i-1]}, v_{[i]}, v_{[i+1]})^T\|},
	$$
	%and $\frac{(v_{[i-1]}, v_{[i]}, v_{[i+1]})^T}{\|(v_{[i-1]}, v_{[i]}, v_{[i+1]})^T\|} \in \text{Comp}_{c_{n}}(a, \kappa)$. 
	it follows that
	$$
	\left\|(U_{i-1}, (D_{i})_{z}, T_{i+1}) \frac{(v_{[i-1]}, v_{[i]}, v_{[i+1]})^T}{\|(v_{[i-1]}, v_{[i]}, v_{[i+1]})^T\|} \right \| \leq t.
	$$
	Let $\CE$ denote the event $\CE_K \cap \left(\cap_{i\in [m]}\{\inf_{w \in \text{Comp}_{c_n}(a, \kappa)}\|M_i w\| > \gamma\}\right)$, where $a, \kappa, \gamma$ are as in Lemma~\ref{lem:comp}. On the event $\CE$, if $t\leq \gamma$, then
	%Hence, by Lemma~\ref{lem:comp}, it follows %that if $t \leq \gamma$, then on the event $\CE_K$, except with probability at most $m \exp(- \gamma b_{n})$, 
	$$\frac{(v_{[i-1]}, v_{[i]}, v_{[i+1]})^T}{\|(v_{[i-1]}, v_{[i]}, v_{[i+1]})^T\|} \in \text{Incomp}_{c_{n}}(a,\kappa).$$
	Therefore, we can conclude from Lemma \ref{lem:incompcoordinates} that on the event $\CE$, %with probability at least $1 - m \exp(-\gamma b_{n})$, 
	any vector $v \in S^{n-1}$ such that $\|X_z v\| \leq \gamma b_{n}^{-10m} m^{-1/2}$ will have at least $\alpha n$ coordinates larger than $\beta b_{n}^{-10 m} m^{-1/2}b_{n}^{-1/2}$, where $\alpha = \gamma'/(K^{2}\log{K})$, $\beta = \gamma'/K$, and $\gamma'>0$ is a constant depending only on $\gamma$. 
	
	Hence, with this choice of $\alpha, \beta, \gamma'$, the probability of the event in \eqref{eqn:smallcoord} is bounded by
	\[\mathbb{P}(\CE_K \cap \CE^{c}) \leq \sum_{i=1}^{m}\Pb\left(\CE_K \cap \{\inf_{w \in \text{Comp}_{c_{n}}(a, \kappa)} \|M_{i} w\| \leq \gamma\}  \right) \leq m\exp(-\gamma b_{n}), \]
	where the last inequality follows by Lemma~\ref{lem:comp}. 
	This proves \eqref{eqn:smallcoord}.

	The next lemma is a direct consequence of Lemma~\ref{lem:incompcoordinates} and Lemmas~2.5 and 2.7 from \cite{cook2018lower}. \begin{lem} 
		\label{lem:anti-concentration-incomp}
		Let $\xi_1,\dots,\xi_k$ be independent copies of a complex random variable $\xi$ satisfying $\E[|\xi|^{2}]=1$. Then, for any $v \in \text{Incomp}_{k}(a, \kappa)$ and for all $\varepsilon \geq 0$,
		$$
		\sup_{r \in \R} \Pb\left(\left|\sum_{i=1}^k v_i \xi_i - r\right| \leq \varepsilon\right) \leq C\kappa^{2}a\left(\varepsilon + \frac{1}{\sqrt{\kappa k}} \right),
		$$
		where $C$ is a constant depending only on $\xi$. 
	\end{lem}
	
	\subsection{Proof of Theorem \ref{thm:lsv}}
	
	\begin{proof}[Proof of Theorem \ref{thm:lsv}]
		By (\ref{eq:decomp}) and \eqref{eqn:smallcoord}, it suffices to bound  
		\begin{align*} 
		\Pb(\CE_K \cap \{\inf_{v \in L_{\alpha,\beta}} \|X_{z} v\| \leq tb_{n}^{-10 m} m^{-1/2}\}), \nonumber
		\end{align*}
		for $t = b_{n}^{-11m}$. 
		By Lemma \ref{lem:distance},
		$$
		\Pb(\CE_K \cap \{\inf_{v \in L_{\alpha,\beta}} \|X_{z} v\| \leq tb^{-10 m} m^{-1/2}\})  \leq  \frac{1}{\alpha} \max_{k \in [n]}\Pb(\CE_K \cap \{\text{dist}(x_k-ze_{k}, \CH_{k}) \leq \beta^{-1}t\}).
		$$
		We will obtain a uniform (in $k$) bound on $\Pb(\CE_K \cap \{\dist(x_k-ze_{k},\CH_{k}) \leq \beta^{-1}t\})$. For convenience of notation, we show this bound for $k=1$. Also, recall from before that we may assume that $x_1-ze_{1}$ is the first row of the matrix, and that $\CH_{1}$ is the span of all the rows except for the first row. 	
		
		Let $\CE$ denote the event that  $$\inf_{w \in \text{Comp}_{c_{n}}(a,\kappa)}\|M_{1}w\| \geq \gamma.$$
		
		Then, by Lemma~\ref{lem:comp}, $\Pb(\CE^{c} \cap \CE_{K}) \leq \exp(-\gamma b_{n})$. Let $\hat n$ denote a unit normal vector to $\CH_{1}$, let $v := (\hat n_{[1]}, \hat n_{[2]}, \hat n_{[m]})$, and let $\widehat{v}:= v/\|v\|$. If $\widehat{v} \in \text{Comp}_{c_{n}}(a,\kappa)$, then on the event $\CE \cap \CE_K$, we have
		\begin{align*}
		|\langle x_1-ze_{1}, \hat n \rangle| 
		&= |\langle x_1-ze_{1}, v \rangle|\\
		&= \|M_{1} v\|\\
		&= \|M_{1}\widehat{v}\|\|v\|\\
		&\geq \gamma\|v\|\\
		&\geq \gamma b_{n}^{-10m} m^{-1/2}.
		\end{align*}
		On the other hand, 
		if $\widehat{v} \in \text{Incomp}_{c_{n}}(a,\kappa)$, then it follows from Lemma~\ref{lem:anti-concentration-incomp} that
		\begin{align*}
		\Pb(|\langle x_1-ze_{1}, \hat n \rangle| \leq \delta ) 
		&= \Pb (|\langle x_1-ze_{1}, v \rangle| \leq \delta)\\
		&= \Pb(|\langle x_1-ze_{1}, \widehat{v} \rangle| \leq \delta/\|v\|)\\
		&\leq C\kappa^{2}a\left(\frac{\delta}{\|v\|} + \frac{1}{\sqrt{\kappa b_{n}}}\right)\\
		&\leq C\kappa^{2}a\left(\delta b^{10m}\sqrt{m} + \frac{1}{\sqrt{\kappa b_{n}}}\right).
		\end{align*}
		Taking $\delta = \beta^{-1}b_{n}^{-11m}$, and combining with the compressible case, we may conclude that
		$$\Pb(\CE_K \cap \{\dist(x_1-ze_{1}, \CH_{1})\leq \beta^{-1}b_{n}^{-11m}\}) \leq C_{K}\frac{1}{\sqrt{b_{n}}}.$$
		The same argument can be used to conclude that
		$$\max_{k \in [n]}\Pb(\CE_K \cap \{\dist(x_k-ze_{k},\CH_{k})\leq \beta^{-1}b_{n}^{-11m}\}) \leq C_{K}\frac{1}{\sqrt{b_{n}}},$$
		which completes the proof. 
	\end{proof}
	
	\section{Convergence of $\nu_{X_{z}}$} \label{sec:esd}
	
	In this section, we establish a rate of convergence for the Stieltjes transform of the empirical eigenvalue distribution of $X_z X_z^*$.

	\begin{thm}\label{thm: JanaSos2017 theorem}
		Let $\tilde X$ be an $n \times n$ periodic block band matrix as defined in Definition \ref{defn: block-band} with atom variable $\xi$. Take $A>1$, and let $z\in \C$ be a fixed complex number. Assume $m_{n,z}(\zeta)=\frac{1}{n}\sum_{i=1}^{n}[\lambda_{i}(X_{z}X_{z}^{*})-\zeta]^{-1}$ is the Stieltjes transform for the empirical spectral measure of $X_{z}X_{z}^{*}$. Suppose that $\xi$ is centered with variance one and $\omega_{4p}:=\E[|\xi|^{4p}]<\infty$ for some integer $p\geq 1$. Then there exists a non random probability measure $\nu_{z}$ on $[0,\infty)$ such that for any $\zeta\in \{\zeta\in \C:-A<\Re(\zeta)<A, 0 < \Im(\zeta)<1\}$
		\begin{align*}
		\E[|m_{n,z}(\zeta)-m_{z}(\zeta)|^{2p}]&\leq \frac{C(p)A^{p}\omega_{4p}}{|\Im(\zeta)|^{8p}}\left[\left(\frac{n}{c_{n}^2}\right)^{p} + \frac{1}{c_n^{p/2}} \right],
		\end{align*}
		where $m_{z}(\zeta)=\int_{\R}\frac{d\nu_{z}(x)}{x-\zeta}$ and $C(p) > 0$ is a constant that depends only on $p$. Moreover, $m_{z}(\zeta)$ is the unique solution to the equation 
		\begin{align} \label{eq:fixedpoint}
		m_{z}(\zeta)=\left[\frac{|z|^{2}}{1+m_{z}(\zeta)}-(1+m_{z}(\zeta))\zeta\right]^{-1},
		\end{align}
		satisfying $\Im(\zeta m_{z}(\zeta^2))>0$ and $\Im(m_z(\zeta)) > 0$ when $\Im(\zeta)>0$. 
	\end{thm}

	\begin{rem}
	We state and prove the above theorem under more general conditions than those of Theorem \ref{thm: main theorem}.  In particular, we allow random variables with no moments above four.  Although, the quantitative estimate improves with the number of existing moments.  
	Furthermore, we do not make use of the lower bound on $c_n$ in Theorem \ref{thm: main theorem}. 
	
	We follow the proof strategy from \cite{jana2017distribution}.  This previous work demonstrated the convergence of the Stieltjes transform for band matrices rather than block band matrices so we necessarily make some adaptations.  More significantly, we deduce an explicit rate of convergence, which does not appear in \cite{jana2017distribution}.
	\end{rem}

	Our main object of study will be 
	\[
	P_{z,\zeta}:=(X_{z}X_{z}^{*})_{\zeta}=(X-zI)(X-zI)^{*}-\zeta I.
	\]  
	Define $X_z^{(k)}$ to be the matrix $X_z$ with the $k$-th column set to zero.  We define
	\begin{align*}
	P^{(k)}_{z,\zeta}&:=(X_{z}^{(k)}X_{z}^{(k)*})-{\zeta}I\\
	&=[(X-zI)-(x_{k}-ze_{k})e_{k}^T][(X-zI)-(x_{k}-ze_{k})e_{k}^T]^{*}-\zeta I\\
	&=(X-zI)(X-zI)^{*}-\zeta I-(x_{k}-ze_{k})(x_{k}-ze_{k})^{*}\\
	&=P_{z,\zeta}-(x_{k}-ze_{k})(x_{k}-ze_{k})^{*}
	\end{align*}
	We also denote
	\[
	m_{n,z}^{(k)}(\zeta):=\frac{1}{n} \tr (P_{z,\zeta}^{(k)})^{-1}.
	\]
	Additionally, we use the shorthand
	\begin{equation} \label{def:alphak}
	\alpha_{k}:=1+(x_{k}-ze_{k})^{*}[P_{z,\zeta}^{(k)}]^{-1}(x_{k}-ze_{k})
	\end{equation} 
	as this term appears repeatedly in our initial calculations. 
	
	For $s_z(\zeta) = m_{n,z}(\zeta)$ or $m_z(\zeta)$, let us define
	\begin{align*}
	f(s_z)&:=\left[\frac{|z|^{2}}{1+s_z(\zeta)}-(1+s_{z}(\zeta))\zeta\right]^{-1}.
	\end{align*}
	
	The motivation for this definition is that $m_z(\zeta)$ is known to be a fixed point of this function when the spectrum obeys the circular law; see Section 11.4 in \cite{bai2010spectral}.
	The proof of Theorem \ref{thm: JanaSos2017 theorem} can be divided into several key computations. 
		Since we expect $m_{n,z}(\zeta)$ to also converge to the fixed point of $f$, we first relate $m_{n,z}(\zeta) - m_z(\zeta)$ to $f(m_{n,z}(\zeta)) - m_{n,z}(\zeta)$.
	\begin{lem}\label{lem:product}
		Under the assumptions of Theorem \ref{thm: JanaSos2017 theorem},
	\begin{align} 
	m_{n,z}(\zeta)-m_{z}(\zeta)=[1-r_{n,z}(\zeta)]^{-1}[m_{n,z}(\zeta)-f(m_{n,z}(\zeta))]
	\end{align}
	where
	\begin{equation} \label{def:rn}
	r_{n,z}(\zeta) = f(m_{n,z}(\zeta))f(m_{z}(\zeta))\left[\frac{|z|^{2}}{(1+m_{n,z}(\zeta))(1+m_{z}(\zeta))}+\zeta\right].
	\end{equation} 
	\end{lem}  
	\begin{proof}
	We have that
	\begin{align} \label{eq:simpledecomp}
	m_{n,z}(\zeta)-m_{z}(\zeta)&=m_{n,z}(\zeta)-f(m_{n,z}(\zeta))+f(m_{n,z}(\zeta))-f(m_{z}(\zeta)),
	\end{align}
	where we have used the fact that $f(m_z(\zeta)) = m_z(\zeta)$, which is known to characterize the circular law; see Section 11.4 and (11.4.1) in \cite{bai2010spectral}.  
	On the other hand,
	\begin{align*}
	f(m_{n,z}(\zeta))&-f(m_{z}(\zeta))\\
	&= f(m_{n,z}(\zeta))f(m_{z}(\zeta)) \left(\frac{1}{f(m_{z}(\zeta))} -\frac{1}{f(m_{n,z}(\zeta))} \right) \\
	&=f(m_{n,z}(\zeta))f(m_{z}(\zeta))\left[\frac{|z|^{2}(m_{z}(\zeta)-m_{n,z}(\zeta))}{(1+m_{n,z}(\zeta))(1+f(m_{z}(\zeta)))}+\zeta(m_{n,z}(\zeta)-m_{z}(\zeta))\right]\\
	&=[m_{n,z}(\zeta)-m_{z}(\zeta)]f(m_{n,z}(\zeta))f(m_{z}(\zeta))\left[\frac{|z|^{2}}{(1+m_{n,z}(\zeta))(1+m_{z}(\zeta))}+\zeta\right]\\
	&=: r_{n,z}(\zeta)[m_{n,z}(\zeta)-m_{z}(\zeta)].
	\end{align*}
	Therefore, by \eqref{eq:simpledecomp},
	\begin{align*}
	m_{n,z}(\zeta)-m_{z}(\zeta)=[1-r_{n,z}(\zeta)]^{-1}[m_{n,z}(\zeta)-f(m_{n,z}(\zeta))]
	\end{align*}
	with $r_{n,z}(\zeta)$ given in \eqref{def:rn}.  
%	\[
%	r_n = f(m_{n,z}(\zeta))f(m_{z}(\zeta))\left[\frac{|z|^{2}}{(1+m_{n,z}(\zeta))(1+m_{z}(\zeta))}+\zeta\right].
%	\]
	\end{proof}
	The strategy of our proof is to control the moments of $m_{n,z}(\zeta)-f(m_{n,z}(\zeta))$ and then provide a deterministic bound for $[1-r_{n,z}(\zeta)]^{-1}$.  
	
	We begin with the moments of $f(m_{n,z}(\zeta)) - m_{n,z}(\zeta)$.
	
	\begin{lem} \label{lem:f-mmoment}
		Under the assumptions of Theorem \ref{thm: JanaSos2017 theorem},
			\begin{align}\label{eqn: estimates of f(m_n)-m_n}
		%\E[|f(m_{n,z}(\zeta))-m_{n,z}(\zeta)|^{2p}]&\leq \frac{C(p)\omega_{4p}}{c_{n}^{p}|\zeta^{-}|^{6p}},\;\;\;\text{under Condition \ref{con: The condition}(I)}\\
		\E[|f(m_{n,z}(\zeta))-m_{n,z}(\zeta)|^{2p}]&\leq \frac{C(p)\omega_{4p}}{| \Im(\zeta) |^{6p}}\left[\left(\frac{n}{c_{n}^2}\right)^{p} + \frac{1}{c_n^{p/2}} \right],%\;\;\;\text{under Condition \ref{con: The condition}(II)},
		\end{align}
	\end{lem}
	\begin{proof}

		%	The proof is organized as follows
		%\begin{enumerate}[(i)]
		%	\item We give an upper bound on $\E[|m_{n,z}(\zeta)-f(m_{n,z}(\zeta))|^{2p}]$.
		%	\item For a sufficiently large $\Im(\zeta)$, $\{m_{n}\}_{n}$ is Cauchy in $L^{2p}$. We estimate a bound on $\E[|m_{n,z}(\zeta)-m_{z}(\zeta)|^{2p}]$.
		%	\item We estimate $\E[|f(m_{n,z}(\zeta))-f(m_{z}(\zeta))|^{2p}]$.
		%	\item Combining the above three, we get the desired bound.
		%%	As a result, we have $m(\zeta)=f(m_{z}(\zeta))$ almost surely. 
		%%	\item Solution to the equation $m(\zeta)=f(m_{z}(\zeta))$ is unique. Since solution to this integral equation is unique, $m$ is non-random.
		%\end{enumerate}	

		We begin by finding a convenient expression to allow us to compute the moments.  
		By the resolvent identity, \footnote{For two invertible matrices $A$ and $B$ of the same dimension, the resolvent identity is the observation that $$A^{-1} - B^{-1} = A^{-1} (B - A) B^{-1}.$$} 
		\begin{align} \label{eq:f-m}
		f(m_{n,z}(\zeta))&I-P_{z,\zeta}^{-1} \nonumber \\ 
		&= f(m_{n,z}(\zeta)) \left[ P_{z,\zeta} - f(m_{n,z}(\zeta))^{-1} I \right] P_{z,\zeta}^{-1} \nonumber\\
		&=f(m_{n,z}(\zeta))\left[(X-zI)(X-zI)^{*}-\frac{|z|^{2}}{1+m_{n,z}(\zeta)}I+\zeta m_{n,z}(\zeta)I\right]P_{z,\zeta}^{-1}.
		%&=f(m_{n,z}(\zeta))\sum_{k=1}^{n}\left[(x_{k}-ze_{k})(x_{k}-ze_{k})^{*}-\frac{1}{n}\frac{|z|^{2}}{1+m_{n,z}(\zeta)}I-\frac{1}{n\alpha_{k}}I\right]P_{z,\zeta}^{-1}.
		\end{align}
		To simplify this expression, we make the following observation.  
		Since $P_{z,\zeta}=X_{z}X_{z}^{*}-\zeta I$, by Lemma \ref{lem: Sherman-Morrison formula},
		\begin{align} \label{eq:minordecomp}
		I+\zeta P_{z,\zeta}^{-1}&=X_{z}X_{z}^{*}P_{z,\zeta}^{-1} \nonumber \\
		&=\sum_{k=1}^{n}(x_{k}-ze_{k})(x_{k}-ze_{k})^{*}P_{z,\zeta}^{-1} \nonumber \\
		&=\sum_{k=1}^{n}(x_{k}-ze_{k})(x_{k}-ze_{k})^{*}[P_{z,\zeta}^{(k)}]^{-1}\alpha_{k}^{-1},
		\end{align}
		where $\alpha_k$ is defined in \eqref{def:alphak}. 
		Taking the normalized trace of (\ref{eq:minordecomp}) yields
		\begin{align*}
		1 + \zeta m_{n,z}(\zeta) &= \frac{1}{n} \sum_{k=1}^n \frac{1}{\alpha_k}\tr((x_k- z e_k) (x_k - z e_k)^* [P_{z,\zeta}^{(k)}]^{-1}) \\
		&= \frac{1}{n} \sum_{k=1}^n \frac{1}{\alpha_k} (x_k - z e_k)^* [P_{z,\zeta}^{(k)}]^{-1} (x_k- z e_k) \\
		&= \frac{1}{n} \sum_{k=1}^n \frac{\alpha_k - 1}{\alpha_k} \\
		&= 1 - \frac{1}{n} \sum_{k=1}^n \frac{1}{\alpha_k}.
		\end{align*}
		From this, we can conclude that
		\begin{align}\label{eq:zetam}
		\zeta m_{n,z}(\zeta)=-\frac{1}{n}\sum_{k=1}^{n}\frac{1}{\alpha_{k}}.
		\end{align}
		Plugging \eqref{eq:zetam} into \eqref{eq:f-m} gives
		\begin{align*}
		f(m_{n,z}(\zeta))&I-P_{z,\zeta}^{-1} \nonumber \\
		&= f(m_{n,z}(\zeta))\left[(X-zI)(X-zI)^{*}-\frac{|z|^{2}}{1+m_{n,z}(\zeta)}I-\frac{1}{n}\sum_{k=1}^{n}\frac{1}{\alpha_{k}} I\right]P_{z,\zeta}^{-1}. 
		\end{align*}
		Taking the normalized trace of this equation we find that
		\begin{align}\label{eqn: difference between normalized traces}
		f(m_{n,z}(\zeta))-m_{n,z}(\zeta)&=\frac{1}{n}f(m_{n,z}(\zeta))\sum_{k=1}^{n}\Big[(x_{k}-ze_{k})^{*}P_{z,\zeta}^{-1}(x_{k}-ze_{k}) \nonumber \\
		&\qquad \qquad \qquad \qquad -\frac{|z|^{2}}{1+m_{n,z}(\zeta)}e_{k}^T P_{z,\zeta}^{-1}e_{k} -\frac{1}{\alpha_{k}}m_{n,z}(\zeta)\Big].
		\end{align}
		We will take the $2p$-th moment of this expression. 
		
		Let us introduce the following notation to organize the terms on the right hand side of (\ref{eqn: difference between normalized traces}).
		Let
		\begin{align*}
		\beta_{k}&:=x_{k}^{*}[P_{z,\zeta}^{(k)}]^{-1}e_{k}, &
		\gamma_{k}&:=e_{k}^T[P_{z,\zeta}^{(k)}]^{-1}x_{k},\\
		\delta_{k}&:=e_{k}^T[P_{z,\zeta}^{(k)}]^{-1}e_{k}, &
		\tau_{k}&:=x_{k}^{*}[P_{z,\zeta}^{(k)}]^{-1}x_{k}.
		\end{align*}
		Recall the definition of $\alpha_k$ given in \eqref{def:alphak}. Since
		\begin{align*}
		\alpha_{k}&=1+(x_{k}-ze_{k})^{*}[P_{z,\zeta}^{(k)}]^{-1}(x_{k}-ze_{k})\\
		&=1+\tau_{k}-z\beta_{k}-\bar{z}\gamma_{k}+|z|^{2}\delta_{k},
		\end{align*}
		again by Lemma \ref{lem: Sherman-Morrison formula}, we can write
		\begin{align*}
		(x_{k}-ze_{k})^{*}P_{z,\zeta}^{-1}(x_{k}-ze_{k})&=\alpha_{k}^{-1}(x_{k}-ze_{k})^{*}[P_{z,\zeta}^{(k)}]^{-1}(x_{k}-ze_{k})\\
		&=\alpha_{k}^{-1}[\tau_{k}-z\beta_{k}-\bar{z}\gamma_{k}+|z|^{2}\delta_{k}]. 
		\end{align*}
		Expanding similarly, 
		\begin{align*}
		e_{k}^TP_{z,\zeta}^{-1}e_{k}&=e_{k}^T[P_{z,\zeta}^{(k)}]^{-1}e_{k}-\alpha_{k}^{-1}e_{k}^T[P_{z,\zeta}^{(k)}]^{-1}(x_{k}-ze_{k})(x_{k}-ze_{k})^{*}[P_{z,\zeta}^{(k)}]^{-1}e_{k}\\
		&=\delta_{k}-\alpha_{k}^{-1}(\gamma_{k}-z\delta_{k})(\beta_{k}-\bar{z}\delta_{k})\\
		&=\alpha_{k}^{-1}[(1+\tau_{k}-z\beta_{k}-\bar{z}\gamma_{k}+|z|^{2}\delta_{k})\delta_{k}-(\gamma_{k}-z\delta_{k})(\beta_{k}-\bar{z}\delta_{k})]\\
		&=\alpha_{k}^{-1}[(1+\tau_{k})\delta_{k}-\gamma_{k}\beta_{k}].
		\end{align*}
		
		Therefore, \eqref{eqn: difference between normalized traces} can be more succinctly written as
		\begin{align}\label{eqn: difference between normalized traces (second)}
		f(m_{n,z}(\zeta))-m_{n,z}(\zeta)&=\frac{1}{n}f(m_{n,z}(\zeta))\sum_{k=1}^{n}\frac{1}{\alpha_{k}}\Big[(\tau_{k}-z\beta_{k}-\bar{z}\gamma_{k}+|z|^{2}\delta_{k})\nonumber\\
		&\qquad -\frac{|z|^{2}}{1+m_{n,z}(\zeta)}\{(1+\tau_{k})\delta_{k}-\gamma_{k}\beta_{k}\}-m_{n,z}(\zeta)\Big]\nonumber\\
		&=\frac{1}{n}f(m_{n,z}(\zeta))\sum_{k=1}^{n}\frac{1}{\alpha_{k}}\Big[(\tau_{k}-m_{n,z}(\zeta))\left\{1-\frac{|z|^{2}\delta_{k}}{1+m_{n,z}(\zeta)}\right\}\nonumber\\
		&\qquad -z\beta_{k}-\bar{z}\gamma_{k}+\frac{|z|^{2}}{1+m_{n,z}(\zeta)}\beta_{k}\gamma_{k}\Big].
		\end{align}
		
		For any $z_1, \dots, z_n \in \C$ and $\ell \in \N$, by Jensen's inequality,
		\begin{equation}\label{eq:Jensens}
		\left|\frac{1}{n} \sum_{i=1}^n z_i\right|^\ell \leq \frac{1}{n} \sum |z_i|^\ell.
		\end{equation}
		As we plan to invoke this inequality, it suffices for our purposes to bound the moment of each summand in \eqref{eqn: difference between normalized traces (second)}.  
		Using Corollary \ref{cor: estimates about band matrices}, 
		\begin{align}
		\E[|\beta_{k}|^{2p}]  &=\E\left[|x_{k}^{*}[P_{z,\zeta}^{(k)}]^{-1}e_{k}e_{k}^T[P_{z,\zeta}^{(k*)}]^{-1}x_{k}|^{p}\right]\nonumber \\ 
		&\leq \frac{2^{p-1}}{c_n^{p/2}} \E \left| (\sqrt{c_n} x_k^* [P_{z, \zeta}^{(k)}]^{-1} e_k e_k^T [P_{z,\zeta}^{(k*)}]^{-1} (\sqrt{c_n} x_k) - \tr\left( [P_{z,\zeta}^{(k)}]^{-1} e_k e_k^T [P_{z,\zeta}^{(k*)}]^{-1} \right)\right|^p \nonumber\\
		&\qquad + \frac{2^{p-1}}{c_n^{p/2}} \left|\tr\left( [P_{z,\zeta}^{(k)}]^{-1} e_k e_k^T [P_{z,\zeta}^{(k*)}]^{-1} \right) \right|^p \nonumber\\
		&\leq C(p) \frac{\omega_{2p}+ 1}{c_{n}^{p/2}|\Im(\zeta)|^{2p}} \nonumber\\
		&\leq \frac{C(p) \omega_{4p}}{c_n^{p/2} |\Im(\zeta)|^{2p}} \label{eqn: estimate of beta}
		\end{align} 
		where $C(p)$ is a constant that only depends on $p$ and may vary from line to line.
		An identical computation yields
		\begin{align}
		\E[|\gamma_{k}|^{2p}]&\leq \frac{C(p)\omega_{4p}}{c_{n}^{p/2}|\Im(\zeta)|^{2p}}\label{eqn: estimate of gamma}.
		\end{align}
		
		By Lemma \ref{lem: difference between traces}, we have
		\begin{align*}
		\left|m_{n,z}(\zeta) - \frac{1}{n}\tr[P_{z,\zeta}^{(k)}]^{-1}\right| = \frac{1}{n} \left|\tr(P_{z,\zeta}^{-1}-[P_{z,\zeta}^{(k)}]^{-1})\right|\leq \frac{1}{n |\Im(\zeta)|}.
		\end{align*}
		Therefore
		\begin{align}\label{eqn: first estimate of the tau}
		&\E\left[\left|\tau_{k}-m_{n,z}(\zeta)\right|^{2p}\right] \nonumber\\
		&\leq 2^{2p} \E\left|\tau_{k}-\frac{1}{n}\tr [P_{z,\zeta}^{(k)}]^{-1}\right|^{2p}+\frac{2^{2p}}{n^{2p}|\Im(\zeta)|^{2p}}\nonumber\\
		&\leq 2^{4p}\E\left|\tau_{k}-\frac{1}{c_{n}}\sum_{i\in I_{k}}[P_{z,\zeta}^{(k)}]^{-1}_{ii}\right|^{2p}+2^{4p}\E\left|\frac{1}{c_{n}}\sum_{i\in I_{k}}[P_{z,\zeta}^{(k)}]^{-1}_{ii}-\frac{1}{n}\tr [P_{z,\zeta}^{(k)}]^{-1}\right|^{2p}+\frac{2^{2p}}{n^{2p}|\Im(\zeta)|^{2p}}.
		\end{align}
		
		We recall that $\tau_{k}=x_{k}^{*}[P_{z,\zeta}^{(k)}]^{-1}x_{k}$, where $x_{k}$ is a band vector already scaled by $1/\sqrt{c_{n}}$. So, from Corollary \ref{cor: estimates about band matrices}, we can conclude that 
		\begin{align*}
		\E\left|\tau_{k}-\frac{1}{c_{n}}\sum_{i\in I_{k}}[P_{z,\zeta}^{(k)}]^{-1}_{ii}\right|^{2p}&\leq \frac{C(p)\omega_{4p}}{c_{n}^{p/2}|\Im(\zeta)|^{2p}}
		\end{align*}
		where $I_k$ denotes the indices in the support of $x_k$.
		
		To estimate the second term of \eqref{eqn: first estimate of the tau}, we use Lemma \ref{lem: rank one perturbation on matrices} to write 
		\begin{align}\label{eqn: difference between partial trace and full trace 1}
		\E\left|\frac{1}{c_{n}}\sum_{i\in I_{k}}[P_{z,\zeta}^{(k)}]^{-1}_{ii}-\frac{1}{n}\tr [P_{z,\zeta}^{(k)}]^{-1}\right|^{2p}&\leq \E\left|\frac{1}{c_{n}}\sum_{i\in I_{k}}[P_{z,\zeta}]^{-1}_{ii}-\frac{1}{n}\tr [P_{z,\zeta}]^{-1}\right|^{2p} + \frac{2^{2p}}{c_n^{2p} \Im(\zeta)^{2p}}.
		\end{align}
		The first expectation on the right hand side can be further decomposed as
		\begin{align}\label{eqn: difference between partial trace and full trace 2}
		&\E\left|\frac{1}{c_{n}}\sum_{i\in I_{k}}[P_{z,\zeta}]^{-1}_{ii}-\frac{1}{n}\tr [P_{z,\zeta}]^{-1}\right|^{2p}\nonumber\\
		&\leq \frac{2^{2p}}{c_{n}^{2p}}\E\left|\sum_{i\in I_{k}}[P_{z,\zeta}]^{-1}_{ii}- \sum_{i\in I_{k}}\E[P_{z,\zeta}]^{-1}_{ii}\right|^{2p}
		+\frac{2^{2p}}{n^{2p}} \E \left|\tr \{P_{z,\zeta}\}^{-1}]-\E \tr[P_{z,\zeta}]^{-1}\right|^{2p}.
		\end{align}
		In the above estimate, we have used the fact that we have a periodic block band matrix with iid entries, therefore $\E[\{P_{z,\zeta}\}^{-1}_{ii}]=\E[\{P_{z,\zeta}\}^{-1}_{11}]$ for all $1\leq i\leq n$, which is the conclusion of Lemma \ref{lem:diagonalexpectation}.  Now, we estimate the first term of \eqref{eqn: difference between partial trace and full trace 2} via a simple martingale decomposition.

		Let $\CF_{k}=\sigma\left(\left\{x_{i}:1\leq i\leq k\right\}\right)$ be the sigma algebra generated by the first $k$ columns of $X$. Let us define
		\begin{align}\label{eqn: rewrite the partial sum as a function of X}
		h(X)=\sum_{i\in I_{k}}[P_{z,\zeta}^{(k)}]_{ii}^{-1}.
		\end{align}
		Then we have the telescoping sum
		\begin{align*}
		h(X)-\E[h(X)]=\sum_{k=1}^{n}\left[\E[h(X)|\CF_{k}]-\E[h(X)|\CF_{k-1}]\right],
		\end{align*}
		where $\CF_{0}$ is the trivial sigma algebra. Using Lemma \ref{lem: rank one perturbation on matrices}, we have $$\left|\E[h(X)|\CF_{k}]-\E[h(X)|\CF_{k-1}]\right|\leq 2/|\Im(\zeta)|.$$ Now by Corollary \ref{cor: azumamoment},
		\begin{align*}
		\E[|h(X)-\E[h(X)]|^{2p}]\leq \frac{C(p)n^{p}}{|\Im(\zeta)|^{4p}},
		\end{align*}
		where $C(p)$ is a constant that depends only on $p$.

		As above, using Lemma \ref{lem: rank one perturbation on matrices} and Result \ref{lem: azuma inequality}, we estimate the second term of \eqref{eqn: difference between partial trace and full trace 2} by
		\begin{align*}
		\E\left|\E[\tr \{P_{z,\zeta}^{(k)}\}^{-1}]-\tr[P_{z,\zeta}^{(k)}]^{-1}\right|^{2p}\leq \frac{C(p)n^{p}}{|\Im(\zeta)|^{4p}}.
		\end{align*}
		
		Using the above estimates in \eqref{eqn: first estimate of the tau}, we obtain
		
		\begin{align}\label{eqn: estimate of tau (final)}
		%\E\left[\left|\tau_{k}-m_{n,z}(\zeta)\right|^{2p}\right]&\leq \frac{C(p)}{c_{n}^{p}\omega^{2p}|\Im(\zeta)|^{4p}},\;\;\;\text{under Condition \ref{con: The condition}(I)}\nonumber\\
		\E\left[\left|\tau_{k}-m_{n,z}(\zeta)\right|^{2p}\right]&\leq \frac{C(p)}{|\Im(\zeta)|^{4p}}\left(\frac{n}{c_{n}^2}\right)^{p},%\;\;\;\text{under Condition \ref{con: The condition}(II)}.
		\end{align}

		To complete the estimates of \eqref{eqn: difference between normalized traces (second)}, we need to lower bound $(f(m_{n,z}(\zeta)))^{-1}$ and $\alpha_{k}$ (recall that $\alpha_k$ is defined in \eqref{def:alphak}).
		
		%We recall that $X=\frac{1}{\sqrt{c_{n}}}\tilde{X}$, where $\tilde{X}$ is a periodic block band matrix. As a result of strong law of large numbers,
		
%		\begin{align*}
%		\frac{1}{nc_{n}}\sum_{j=1}^{n}\sum_{i\in I_{j}}x_{ij}^{2}\stackrel{a.s.}{\to}1,\;\;\text{as $n\to\infty$}.
%		\end{align*}
		
		Since $\Im(\zeta) > 0$, it follows that 
		\begin{align*}
		\delta:= \int_{0}^{\infty}\frac{1}{|\lambda-\zeta|^{2}}\;d\mu_{X_{z}X_{z}^{*}}(\lambda)>0.
		\end{align*}
		As a result, for any $\zeta\in \C$ with $\Im(\zeta)>0$,
		\begin{align*}
		\Im(\zeta m_{n,z}(\zeta))&=\int_{0}^{\infty}\frac{\Im(\zeta)\lambda}{|\lambda-\zeta|^{2}}\;d\mu_{X_{z}X_{z}^{*}}(\lambda)\geq 0\\
		%\Re(\zeta m_{n,z}(\zeta))&=\int_{0}^{\infty}\frac{\Re(\zeta)\lambda-|\zeta|^{2}}{|\lambda-\zeta|^{2}}\;d\mu_{X_{z}X_{z}^{*}}(\lambda)\leq 0,\;\text{(if $\Re(\zeta)\leq 0$)}\\
		\Im(m_{n,z}(\zeta))&=\int_{0}^{\infty}\frac{\Im(\zeta)}{|\lambda-\zeta|^{2}}\;d\mu_{X_{z}X_{z}^{*}}(\lambda)\geq \Im(\zeta)\delta>0.
		\end{align*}
%		For $\Im(\zeta) \leq 0$, it is straightforward to deduce that
%		\[
%		\Im(\zeta m_{n,z}(\zeta)) \leq 0, \qquad \Re(\zeta m_{n,z}(\zeta)) \leq 0 \,(\text{if } \Re(\zeta) \leq 0) \text{ and } \Im(m_{n,z}(\zeta)) < 0.\]
		Using the above estimates, we have 
		
		\begin{align*}
		|\Im(f(m_{n,z}(\zeta))^{-1})|&=\left|\Im\left[\frac{|z|^{2}}{1+m_{n,z}(\zeta)}-(1+m_{n,z}(\zeta))\zeta\right]\right|\\
		&=\left|\left[\frac{|z|^{2}\Im(\overline{ {m}_{n,z}(\zeta)})}{|1+m_{n,z}(\zeta)|^{2}}-\Im(\zeta)-\Im(\zeta m_{n,z}(\zeta))\right]\right|\\
		&\geq |\Im(\zeta)|.
		\end{align*}

		Therefore
		\begin{equation} \label{eq: lowerf}
		|f(m_{n,z}(\zeta))^{-1}| \geq |\Im(f(m_{n,z}(\zeta))^{-1})| \geq |\Im(\zeta)|.
		\end{equation}

		Following the similar computation as \eqref{eqn: lower bound of alpha_k}, we can also conclude that
		\begin{align} \label{eq: loweralpha}
		|\alpha_{k}|\geq \delta|\Im(\zeta)|.
		\end{align}
		Finally, plugging \eqref{eq: loweralpha}, \eqref{eq: lowerf}, \eqref{eqn: estimate of beta}, \eqref{eqn: estimate of gamma}, \eqref{eqn: estimate of tau (final)} into \eqref{eqn: difference between normalized traces (second)} gives the desired bound \eqref{eqn: estimates of f(m_n)-m_n}. 
		%\begin{align*}
		%\E[|f(m_{n,z}(\zeta))-m_{n,z}(\zeta)|^{2p}]&\leq \frac{C(p)\omega_{4p}}{c_{n}^{p}|\Im(\zeta)|^{6p}},\;\;\;\text{under Condition \ref{con: The condition}(I)}\\
		%\E[|f(m_{n,z}(\zeta))-m_{n,z}(\zeta)|^{2p}]&\leq \frac{C(p)\omega_{4p}}{|\Im(\zeta)|^{6p}}\left(\frac{n}{c_{n}^{2}}\right)^{p},%\;\;\;\text{under Condition \ref{con: The condition}(II)},
		%\end{align*}
	\end{proof}
	
	Next, we provide a deterministic upper bound on $| 1- r_{n,z}(\zeta)|$.  
	\begin{lem} \label{lem:rbound}
		Under the assumptions of Theorem \ref{thm: JanaSos2017 theorem},
		\begin{equation}
		|1-r_{n,z}(\zeta)| \geq \frac{|\Im(\zeta)|}{4\sqrt{A}}.
		\end{equation}
	\end{lem}
	\begin{proof}
		
		Let us denote 
		\begin{align*}
		A_{n,z}(\zeta)&:=1+ m_{n,z}(\zeta)&A_{z}(\zeta)&:=1+m_{z}(\zeta)\\
		B_{n,z}(\zeta)&:=|z|^{2}- \zeta A_{n,z}(\zeta)^{2}&B_{z}(\zeta)&:=|z|^{2}-\zeta  A_{z}(\zeta)^{2}\\
		\epsilon_{n,z}(\zeta)&:=m_{n,z}(\zeta)-f(m_{n,z}(\zeta)).
		\end{align*}
		Let $m_{z}(\zeta)$ be the solution of the equation $m_{z}(\zeta)=A_{z}(\zeta)/B_{z}(\zeta)$ satisfying $\Im(\sqrt{\zeta}m_{z}(\zeta))>0$ when $\Im(\sqrt{\zeta})>0$, where we have used the negative real axis for the branch cut of the square root function. The existence of such a solution is well-known in the circular law literature (see Section 11.4 in \cite{bai2010spectral}).
		
		Observe that as per the above notations, we may write
		\begin{align*}
		f(m_{n,z}(\zeta))&=\frac{A_{n,z}(\zeta)}{B_{n,z}(\zeta)},\\
		m_{n,z}(\zeta)&=\frac{A_{n,z}(\zeta)}{B_{n,z}(\zeta)}+\epsilon_{n,z}(\zeta).
		\end{align*}
		Using the fact that $|ab|\leq \frac{1}{2}(|a|^{2}+|b|^{2})$ for $a,b \in \C$, and employing a similar calculation as in
		 \cite{gotze2010circular}, we write
		\begin{align}\label{eqn: estimate of r_n}
		|1-r_{n,z}(\zeta)|&=\left|1-\frac{|z|^{2}+\zeta A_{z}(\zeta)A_{n,z}(\zeta)}{B_{z}(\zeta)B_{n,z}(\zeta)}\right|\nonumber\\
		&\geq 1 - \left|\frac{|z|^{2}+\zeta A_{z}(\zeta)A_{n,z}(\zeta)}{B_{z}(\zeta)B_{n,z}(\zeta)}\right| \nonumber\\
		&\geq 1 - \frac{|z|^{2}+|\zeta A_{z}(\zeta)A_{n,z}(\zeta)|}{|B_{z}(\zeta)B_{n,z}(\zeta)|} \nonumber\\
		&\geq \frac{1}{2}\left(1-\frac{|z|^{2}+|\sqrt{\zeta}A_{z}(\zeta)|^{2}}{|B_{z}(\zeta)|^{2}}\right)+\frac{1}{2}\left(1-\frac{|z|^{2}+|\sqrt{\zeta}A_{n,z}(\zeta)|^{2}}{|B_{n,z}(\zeta)|^{2}}\right).
		%\geq\frac{1}{2}\left(1-\frac{|z|^{2}|f(m_{n,z}(\zeta))|^{2}}{|1+m_{n,z}(\zeta)|^{2}}-|\zeta| |f(m_{n,z}(\zeta))|^{2}\right)\\
		%&+\frac{1}{2}\left(1-\frac{|z|^{2}|f(m_{z}(\zeta))|^{2}}{|1+m_{z}(\zeta)|^{2}}-|\zeta| |f(m_{z}(\zeta))|^{2}\right)
		\end{align}
		Now, we estimate lower bounds for each expression of \eqref{eqn: estimate of r_n}. We proceed as follows:
		\begin{align*}
		\Im(\sqrt{\zeta} A_{n,z}(\zeta))&=\Im(\sqrt{\zeta} m_{n,z}(\zeta))+\Im(\sqrt{\zeta})\\
		&= \Im \left[ \frac{\sqrt{\zeta} A_n,z (\zeta) \bar{B}_{n,z}(\zeta)}{|B_{n,z}(\zeta)|^2}\right] + \Im(\sqrt{\zeta} \epsilon_{n,z}(\zeta)) + \Im(\sqrt{\zeta}) \\
		&= \Im \left[ \frac{\sqrt{\zeta} A_{n,z}(\zeta) (|z|^2 - \overline{\zeta A_{n,z}(\zeta)^2}}{|B_{n,z}(\zeta)|^2}\right] + \Im(\sqrt{\zeta} \epsilon_{n,z}(\zeta)) + \Im(\sqrt{\zeta}) \\
		&=  \frac{\Im\big[\sqrt{\zeta} A_{n,z}(\zeta) |z|^2 - |\sqrt{\zeta} A_{n,z}(\zeta)|^2 \overline{\sqrt{\zeta} A_{n,z}(\zeta)}\big]}{|B_{n,z}(\zeta)|^2} + \Im(\sqrt{\zeta} \epsilon_{n,z}(\zeta)) + \Im(\sqrt{\zeta}) \\
		&=\Im(\sqrt{\zeta} A_{n,z}(\zeta))\left[\frac{|z|^{2}+|\sqrt{\zeta}A_{n,z}(\zeta)|^{2}}{|B_{n,z}(\zeta)|^{2}}\right]+\Im(\sqrt{\zeta}\epsilon_{n,z}(\zeta))+\Im(\sqrt{\zeta}).
		\end{align*}
		
		Consequently, we have
		\begin{align}\label{eqn: estimate of r_n_part 1(first)}
		1-\frac{|z|^{2}+|\sqrt{\zeta}A_{n,z}(\zeta)|^{2}}{|B_{n,z}(\zeta)|^{2}}&=\frac{\Im(\sqrt{\zeta}\epsilon_{n,z}(\zeta))+\Im(\sqrt{\zeta})}{\Im(\sqrt{\zeta} A_{n,z}(\zeta))}=\frac{\Im(\sqrt{\zeta}\epsilon_{n,z}(\zeta))+\Im(\sqrt{\zeta})}{\Im(\sqrt{\zeta})+\Im(\sqrt{\zeta} m_{n,z}(\zeta))}.
		\end{align}
		Similarly,
		\begin{align}\label{eqn: estimate of r_n_part 2(first)}
		1-\frac{|z|^{2}+|\sqrt{\zeta}A_{z}(\zeta)|^{2}}{|B_{z}(\zeta)|^{2}}&=\frac{\Im(\sqrt{\zeta})}{\Im(\sqrt{\zeta} A_{z}(\zeta))}=\frac{\Im(\sqrt{\zeta})}{\Im(\sqrt{\zeta})+\Im(\sqrt{\zeta} m_{z}(\zeta))}.
		\end{align}
		Recall that we have chosen the solution $m_z(\zeta)$ such that $\Im(\sqrt{\zeta}m_{z}(\zeta))$ and $\Im(\sqrt{\zeta})$ have the same sign. Therefore,
		\begin{align*}
		0&\leq \frac{\Im(\sqrt{\zeta})}{\Im(\sqrt{\zeta})+\Im(\sqrt{\zeta} m_{z}(\zeta))}
		=1-\frac{|z|^{2}+|\sqrt{\zeta}A_{z}(\zeta)|^{2}}{|B_{z}(\zeta)|^{2}}
		=1-\frac{|z|^{2}}{|B_{z}(\zeta)|^{2}}-|\sqrt{\zeta}m_{z}(\zeta)|^{2}.
		\end{align*}
		As a result,
		\begin{align*}
		|\sqrt{\zeta}m_{z}(\zeta)|\leq 1.
		\end{align*} 
		
		Using the the above estimate in \eqref{eqn: estimate of r_n_part 2(first)} and the fact that $\Im(\sqrt{\zeta}m_{z}(\zeta))$ and $\Im(\sqrt{\zeta})$ have the same sign, we obtain
		\begin{align}\label{eqn: estimate of r_n_part 2(final)}
		1-\frac{|z|^{2}+|\sqrt{\zeta}A_{z}(\zeta)|^{2}}{|B_{z}(\zeta)|^{2}}
		&=\frac{\Im(\sqrt{\zeta})}{\Im(\sqrt{\zeta})+\Im(\sqrt{\zeta} m_{z}(\zeta))}\nonumber\\
		&=\frac{|\Im(\sqrt{\zeta})|}{|\Im(\sqrt{\zeta})|+|\Im(\sqrt{\zeta} m_{z}(\zeta))|}\nonumber\\
		&\geq \frac{|\Im(\sqrt{\zeta})|}{|\Im(\sqrt{\zeta})|+1}\nonumber\\
		&=\frac{1}{1+|\Im(\sqrt{\zeta})|^{-1}}\nonumber\\
		&\geq \frac{1}{1+3\sqrt{A}|\Im(\zeta)|^{-1}}\nonumber\\
		&\geq \frac{|\Im(\zeta)|}{4\sqrt{A}},
		\end{align}
		where the second to last inequality follows from the fact that $|\Im(\sqrt{\zeta})|>|\Im(\zeta)|/3\sqrt{A}$ which is implied by the assumption $\zeta\in \{\zeta\in \C:-A<\Re(\zeta)<A,0 < \Im(\zeta)<1\}$ and $A>1$.
		
%		To prove \eqref{eqn: estimate of r_n_part 1(first)}, we shall first show that $\Im(\sqrt{\zeta}\epsilon_{n,z}(\zeta))$ has the same sign as $\Im(\sqrt{\zeta})$. We notice that $\Im(\sqrt{\zeta}m_{n,z}(\zeta))=\Im(\sqrt{\zeta})\int_{0}^{\infty}\frac{\lambda+|\zeta|}{|\lambda-\zeta|^{2}}\;d\mu_{X_{z}X_{z}^{*}}(\lambda)$ has the same sign as $\Im(\sqrt{\zeta})$. Secondly,
%		\begin{align*}
%		|B_{n,z}(\zeta)|&=\left||z|-\sqrt{\zeta}A_{n,z}(\zeta)\right|\left||z|+\sqrt{\zeta}A_{n,z}(\zeta)\right|\\
%		&\geq \left|\Im\left(|z|-\sqrt{\zeta}A_{n,z}(\zeta)\right)\right|\left|\Im\left(|z|+\sqrt{\zeta}A_{n,z}(\zeta)\right)\right|\\
%		&=\left|\Im(\sqrt{\zeta})+\Im(\sqrt{\zeta}m_{n,z}(\zeta))\right|^{2}\\
%		&\geq |\Im(\sqrt{\zeta})|^{2}
%		\end{align*}
		
		Similarly,
		
		\begin{align}\label{eqn: estimate of r_n_part 1(final)}
		1-\frac{|z|^{2}+|\sqrt{\zeta}A_{n,z}(\zeta)|^{2}}{|B_{n,z}(\zeta)|^{2}}%&=\frac{\Im(\sqrt{\zeta}\epsilon_{n,z}(\zeta))+\Im(\sqrt{\zeta})}{\Im(\sqrt{\zeta})+\Im(\sqrt{\zeta} m_{n,z}(\zeta))}\geq \frac{|\Im(\sqrt{\zeta})|}{2}
		\geq\frac{|\Im(\zeta)|}{4\sqrt{A}}.
		\end{align}
		
		Using the estimates \eqref{eqn: estimate of r_n_part 1(final)}, \eqref{eqn: estimate of r_n_part 2(final)} in \eqref{eqn: estimate of r_n}, we have
		
		\begin{align*}
		|1-r_{n,z}(\zeta)|\geq \frac{|\Im(\zeta)|}{4\sqrt{A}}.
		\end{align*}

%		Using the above estimates in \eqref{eqn: estimate of r_n} and the fact that $|\Im(\zeta)|<1/4$, we obtain
%		\begin{align*}
%		|1-r_{n,z}(\zeta)|&\geq \frac{1}{2}\left[\frac{\Im(\sqrt{\zeta}\epsilon_{n,z}(\zeta))+\Im(\sqrt{\zeta})}{\Im(\sqrt{\zeta})+\Im(\sqrt{\zeta} m_{n,z}(\zeta))}+\frac{\Im(\sqrt{\zeta})}{\Im(\sqrt{\zeta})+\Im(\sqrt{\zeta} m_{z}(\zeta))}\right]\geq \frac{|\Im(\sqrt \zeta)|}{4}\geq \frac{|\Im(\zeta)|^{2}}{4},\\
%		&
%		\end{align*} 
%		where the last inequality follows from the fact that
%		\begin{align*}
%		|\Im(\sqrt{\zeta})|=\sqrt{r}|\sin(\theta/2)|\geq 
%		r^{2}|\sin \theta|^{2}=|\Im(\zeta)|^{2},\;\;\;\text{ when $|r\sin\theta|<1/4$. }
%		\end{align*} 

	\end{proof}

	Theorem \ref{thm: JanaSos2017 theorem} follows easily from the above calculations.  
	\begin{proof}[Proof of Theorem \ref{thm: JanaSos2017 theorem}]
		By Lemma \ref{lem:product},
		\begin{equation}
		\E[|m_{n,z}(\zeta) - m_z(\zeta)|^{2p}] = \E|[1-r_{n,z}(\zeta)]^{-1}[m_{n,z}(\zeta)-f(m_{n,z}(\zeta))]|^{2p}
		\end{equation}
		Therefore, by Lemmas \ref{lem:f-mmoment} and \ref{lem:rbound}, 
		\begin{align*}
		%\E[|m_{n,z}(\zeta)-m_{z}(\zeta)|^{2p}]&\leq \frac{C(p)\omega_{4p}}{c_{n}^{p}|\zeta^{-}|^{8p}},\;\;\;\text{under Condition \ref{con: The condition}(I)}\\
		\E[|m_{n,z}(\zeta)-m_{z}(\zeta)|^{2p}]&\leq \frac{C(p)A^{p}\omega_{4p}}{|\Im(\zeta)|^{8p}}\left[\left(\frac{n}{c_{n}^2}\right)^{p} + \frac{1}{c_n^{p/2}} \right].%\;\;\;\text{under Condition \ref{con: The condition}(II)}.
		\end{align*}
		
		Since $m_z$ is the Stieltjes transform of $\nu_z$, it is a well-known property (see, for example, Section 11.4 and (11.4.1) in \cite{bai2010spectral}) that $m_z$ is the unique solution of \eqref{eq:fixedpoint}  satisfying $\Im(\zeta m_{z}(\zeta^2))>0$ and $\Im(m_z(\zeta)) > 0$ when $\Im(\zeta)>0$.  
		%The last statement in Theorem \ref{thm: JanaSos2017 theorem} is a well-known property of the Stieltjes transform of the circular law (see \cite[Section 11.4]{bai2010spectral}).
	\end{proof}

	\section{Proof of Theorem \ref{thm: main theorem}}\label{sec: proof of the main theorem}

	\subsection{Spectral norm bound}
	Before proving Theorem \ref{thm: main theorem}, we note the following spectral norm bound on $X$.

	\begin{prop}[Spectral norm bound] \label{prop:norm}
		There exists a constant $K > 0$ such that $\|X\| \leq K$ with probability $1- o(1)$.  
	\end{prop} 
	\begin{proof}
		For any vector $v \in \mathbb{C}^n$, it follows from the block structure of $X$ that 
		\[ \| X v \| \leq C \|v\| \left( \max_{1 \leq i \leq m} \|T_i\| + \max_{1 \leq i \leq m} \|U_i\| + \max_{1 \leq i \leq m} \|D_i\| \right), \]
		where $C > 0$ is an absolute constant.  The claim then follows from Lemma \ref{lem:EKc}.  
	\end{proof}

	\subsection{Proof of Theorem \ref{thm: main theorem}}
	
	In order to complete the proof of Theorem \ref{thm: main theorem}, we will use the following replacement principle from \cite{tao2010random}.  Let $\|A\|_2$ denote the Hilbert--Schmidt norm of the matrix $A$ defined by the formula
	\[ \|A \|_2 := \sqrt{ \tr(A A^\ast) } = \sqrt{ \tr (A^\ast A) }. \]

	\begin{thm}[Replacement principle; Theorem 2.1 from \cite{tao2010random}] \label{thm:replacement}
		Suppose for each $n$ that $G$ and $X$ are $n \times n$ ensembles of random matrices.  Assume that: 
		\begin{enumerate}[(i)]
			\item \label{item:lln} the expression
			\[ \frac{1}{n} \|G\|_2^2 + \frac{1}{n} \|X\|_2^2 \]
			is bounded in probability (resp. almost surely);
			\item for almost all complex numbers $z$,
			\[ \frac{1}{n} \log \left| \det \left(  G_z \right) \right| - \frac{1}{n} \log \left| \det \left(  X_{z} \right) \right| \]
			converges in probability (resp. almost surely) to zero and, in particular, for fixed $z$, these determinants are nonzero with probability $1 - o(1)$ (resp. almost surely nonzero for all but finitely many $n$).  
		\end{enumerate}
		Then 
		\[ \mu_{G} - \mu_{ X} \]
		converges in probability (resp. almost surely) to zero.
	\end{thm}

	We will apply the replacement principle to the normalized band matrix $X$, while the other matrix is taken to be $G := \frac{1}{\sqrt{n}} {\tilde G}$, where the entries of the $n \times n$ matrix $\tilde G$ are iid standard Gaussian random variables, i.e., $\tilde G$ is a Ginibre matrix.  As the limiting behavior of $\mu_{G}$ is known to be almost surely the circular law \cite{tao2010random}, it will suffice, in order to complete the proof of Theorem \ref{thm: main theorem}, to check the two conditions of Theorem \ref{thm:replacement}.  
	
	Condition \eqref{item:lln} from Theorem \ref{thm:replacement} follows by the law of large numbers.  Thus, it suffices to verify the second condition.  To do so, we introduce the following notation inspired by Chapter 11 of \cite{bai2010spectral}.  For $z \in \mathbb{C}$, we define the following empirical distributions constructed from the squared singular values of $X_{z}$ and $G_{z}$:
	\[ \nu_{X_{z}}(\cdot) := \frac{1}{n} \sum_{i=1}^n \delta_{s_i^2(X_{z})}(\cdot) \]
	and
	\[ \nu_{G_{z}}(\cdot) := \frac{1}{n} \sum_{i=1}^n \delta_{s_i^2(G_{z})}(\cdot). \] 
	It follows that
	\begin{align*}
	\frac{1}{n} \log \left| \det \left(X_{z} \right) \right| - \frac{1}{n} \log \left| \det \left(G_{z}\right) \right| = \frac{1}{2} \int_{0}^\infty \log x \  \nu_{X_{z}}(dx) - \frac{1}{2} \int_{0}^\infty \log x \  \nu_{G_{z}}(dx).  
	\end{align*}

	By Theorem \ref{thm:lsv} as well as Proposition \ref{prop:norm}, there exists a constant $K > 0$ (depending on $z$) such that 
	\begin{equation} \label{eq:prebound}
	\int_{0}^\infty \log x \  \nu_{X_{z}}(dx) - \int_{0}^\infty \log x \  \nu_{G_{z}}(dx) = \int_{c_n^{-25m}}^K \log x \   \nu_{X_{z}}(dx) - \int_{c_n^{-25m}}^K \log x \  \nu_{G_{z}}(dx) 
	\end{equation}
	with probability $1-o(1)$.  Here, the largest and smallest singular values of $G_{z}$ can be controlled by the results in \cite{tao2008random, vershynin2012}.  
	We will apply the following lemma.

	\begin{lem} \label{lemma:ibp}
		For any probability measure $\mu$ and $\nu$ on $\mathbb{R}$ and any $0 < a < b$,
		\[ \left| \int_{a}^b \log(x) d \mu(x) - \int_a^b \log(x) d \nu(x) \right| \leq 2 [ | \log b | + |\log a| ] \| \mu - \nu \|_{[a,b]}, \]
		where
		\[ \| \mu - \nu \|_{[a,b]} := \sup_{x \in [a,b]} | \mu( [a,x]) - \nu([a,x]) |. \]
	\end{lem}
	\begin{proof}
		We rewrite
		\begin{align*}
		\int_a^b \log(x) d \mu(x) = \log(b) \mu([a,b]) - \int_{a}^b \int_x ^b \frac{1}{t} dt d\mu(x). 
		\end{align*}
		Applying Fubini's theorem, we deduce that
		\[ \int_{a}^b \int_x ^b \frac{1}{t} dt d\mu(x) = \int_a^b \frac{ \mu([a,t]) }{t} dt. \]
		Similarly, the same equalities apply to $\nu$.  Thus, we obtain that
		\begin{align*}
		&\left| \int_{a}^b \log(x) d \mu(x) - \int_a^b \log(x) d \nu(x) \right| \\
		&\leq |\log (b)| | \mu([a,b]) - \nu([a,b]) | + \left| \int_a^b \frac{ \mu([a,t]) - \nu([a,t]) }{t} dt \right| \\
		&\leq | \log b| \| \mu - \nu \|_{[a,b]} + \| \mu - \nu \|_{[a,b]} \int_{a}^b \frac{1}{t} dt,
		\end{align*}
		from which the conclusion follows.  
	\end{proof}
	
	Returning to \eqref{eq:prebound} and applying the above lemma, we find that
	\begin{equation} \label{eq:svn}
	\left| \frac{1}{n} \log \left| \det \left(X_{z}\right) \right| - \frac{1}{n} \log \left| \det \left(G_{z}\right) \right| \right|  \leq C \frac{n}{b_n} \log (n)  \| \nu_{X_{z}}(\cdot) - \nu_{G_{z}}(\cdot) \|_{[0, \infty)} 
	\end{equation} 
	for a constant $C > 0$, where 
	\[ \| \mu - \nu \|_{[0, \infty)} := \sup_{x \geq 0} | \mu ( (-\infty, x]) - \nu( (-\infty, x]) | \]  
	for any probability measures $\mu$ and $\nu$ on $\mathbb{R}$.  
	Let $\nu_{z}(\cdot)$ be the probability measure on $[0, \infty)$ from Theorem \ref{thm: JanaSos2017 theorem} (or equivalently, the probability measure defined in Section 11.4 of \cite{bai2010spectral}).  By the triangle inequality, it suffices to show that
	\begin{equation} \label{eq:finish1}
	\| \nu_{X_{z}}(\cdot) - \nu_{z}(\cdot) \|_{[0, \infty)} = O \left( \left( \frac{ n \log n}{ b_n^2} \right)^{1/31} \right) 
	\end{equation}
	and 
	\begin{equation} \label{eq:finish2}
	\| \nu_{G_{z}}(\cdot) - \nu_{z}(\cdot) \|_{[0, \infty)} = O \left( \left( \frac{ n \log n}{ b_n^2} \right)^{1/31} \right) 
	\end{equation}
	with probability $1 - o(1)$.  
	The convergence in \eqref{eq:finish2} follows from Lemma 11.16 from \cite{bai2010spectral}; in fact, the results in \cite{bai2010spectral} provide a much better error bound which holds almost surely.  Thus, it remains to establish \eqref{eq:finish1}, which is a consequence of the following lemma.  
	
	\begin{lem} \label{lemma:rate}
		Let $\tilde X$ and $X$ be as in Theorem \ref{thm: main theorem} with $b_n \geq n^{32/33} \log n$.  Then, for any fixed $z \in \mathbb{C}$, 
		\[ \| \nu_{X_{z}}(\cdot) - \nu_{z}(\cdot) \|_{[0, \infty)} = O \left( \left( \frac{ n \log n}{ b_n^2} \right)^{1/31} \right)  \] 
		with probability $1 - o(1)$.  
	\end{lem}
	\begin{proof}
		Fix $z \in \mathbb{C}$.  For notational simplicity define 
		\[ q_n := \frac{ n \log n}{ b_n^2}. \]
		Let $m_{n,z}$ be the Stieltjes transform of $\nu_{X_{z}}(\cdot)$ and $m_z$ be the Stieltjes transform of $\nu_{z}(\cdot)$.  We consider both Stieltjes transforms only on the upper-half plane $\mathbb{C}^+$.  On the upper-half plane, both Stieltjes transforms are Lipschitz:
		\begin{equation} \label{eq:lip}
		|m_{n,z}(\zeta) - m_{n,z}(\xi)| \leq \frac{|\zeta - \xi|}{\Im \zeta \Im \xi}, \qquad |m_{z}(\zeta) - m_{z}(\xi)| \leq \frac{|\zeta - \xi|}{\Im \zeta \Im \xi}. 
		\end{equation}
		
		Fix $A > 0$ sufficiently large to be chosen later.  Define the line segment in the complex plane:
		\begin{equation} \label{eq:defL}
		L := \left\{ \zeta = \theta + i q_n^{2/31} \in \mathbb{C}^+ : - A \leq \theta \leq A \right\}.  
		\end{equation}
		Applying Theorem \ref{thm: JanaSos2017 theorem} and Markov's inequality, for any $\zeta \in L$, we have
		\[ \Prob \left( |m_{n,z}(\zeta) - m_z(\zeta) | \geq q_n^{5/31} \right) \leq \frac{C}{q_n^{26/31}} \frac{n} {b_n^2}  \]
		for a constant $C > 0$ which depends only on the moments of the atom variable $\xi$ and $A$.  Let $\mathcal{N}$ be a $q_n^{5/31}$-net of $L$.  By a simple covering argument, $\mathcal{N}$ can be chosen so that $|\mathcal{N}| = O(q_n^{-5/31})$.  Thus, by the union bound, 
		\[ \Prob \left(\sup_{\zeta \in \mathcal{N}} |m_{n,z}(\zeta) - m_z(\zeta) | \geq q_n^{5/31} \right) \leq \frac{C}{q_n} \frac{n} {b_n^2} = \frac{C}{\log n} = o(1).  \]
		Using the Lipschitz continuity \eqref{eq:lip}, this bound can be extended to all of $L$, and we obtain
		\begin{equation} \label{eq:mnmz}
		\sup_{\zeta \in L} |m_{n,z}(\zeta) - m_z(\zeta) | = O( q_n^{1/31} ). 
		\end{equation} 
		with probability $1 - o(1)$.  
		
		To complete the proof, we will use Corollary B.15 from \cite{bai2010spectral} and \eqref{eq:mnmz} to bound $\| \nu_{X_{z}}(\cdot) - \nu_z(\cdot) \|_{[0, \infty)}$. 
		Indeed, take $K > 0$ sufficiently large so that $\nu_{X_{z}}([0, K]) = 1$ with probability $1-o(1)$ and $\nu_z([0, K]) = 1$.  Such a choice is always possible by Proposition \ref{prop:norm} and since $\nu_z$ has compact support (a fact which can also be deduced from Proposition \ref{prop:norm}).  Recall the parameter $A > 0$ used to define the line segment $L$ (see \eqref{eq:defL}).  Taking $A, a > 0$ sufficiently large, setting $\eta := q_n^{2/31}$, and letting $\zeta := \theta + i \eta$, Corollary B.15 from \cite{bai2010spectral} implies that
		\begin{align*}
		&\| \nu_{X_{z}}(\cdot) - \nu_{z}(\cdot) \|_{[0, \infty)}\\
		& \leq C \left[ \int_{-A}^A |m_{n,z}(\zeta) - m_z(\zeta) | d \theta + \frac{1}{\eta} \sup_{x} \int_{|y| \leq 2 \eta a} | \nu_z( (-\infty, x+y]) - \nu_z((-\infty, x]) | dy \right], 
		\end{align*}
		where $C > 0$ depends only on the choice of $A, K , a$.  The second term is bounded by Lemma 11.9 from \cite{bai2010spectral}: 
		\[ \frac{1}{\eta} \sup_{x} \int_{|y| \leq 2 \eta a} | \nu_z( (-\infty, x+y]) - \nu_z((\infty, x]) | dy \leq C' \sqrt{\eta} \]
		for a constant $C' > 0$ depending only on $a$.  For the first term we apply \eqref{eq:mnmz} to obtain 
		\[ \int_{-A}^A |m_{n,z}(\zeta) - m_z(\zeta) | d \theta =  O \left( q_n^{1/31} \right) \]
		with probability $1-o(1)$.  Combining the two bounds above, we conclude that, with probability $1-o(1)$, 
		\[ \| \nu_{X_{z}}(\cdot) - \nu(\cdot, z) \|_{[0, \infty)} = O \left( q_n^{1/31} \right), \]
		which completes the proof of the lemma.  
	\end{proof}

	Lemma \ref{lemma:rate} establish \eqref{eq:finish1}.   Combining \eqref{eq:finish1}, \eqref{eq:finish2} with \eqref{eq:svn} and taking $b_n \geq n^{32/33} \log n$ completes the proof of Theorem \ref{thm: main theorem}.

	\appendix
	\renewcommand{\theequation}{\thesection\arabic{equation}}
	\setcounter{equation}{0}
	\section{Auxiliary tools}
	\begin{lem}[Sherman-Morrison formula; see Section 0.7.4 in \cite{MR2978290}]\label{lem: Sherman-Morrison formula}
		Let $A$ and $A+vv^{*}$ be two invertible matrices, where $v\in \C^{n}$. Then
		\begin{align*}
		v^{*}(A+vv^{*})^{-1}=\frac{v^{*}A^{-1}}{1+v^{*}A^{-1}v}.
		\end{align*}
	\end{lem}
	
%	\begin{lem}
%		Let $P$ be a non-negative definite matrix and $\zeta\in \C\backslash\R_{+}$, then $\|(P-\zeta I)^{-1}\|\leq 1/|\Im(\zeta)|$.
%	\end{lem}
	
	\begin{lem}\label{lem: difference between traces}
		Let $\zeta\in \C\backslash\R_{+}$, and $A$ be an $n\times n$ non-negative definite matrix. Then for any $v\in \C^{n}$,
		\begin{align*}
		|\tr[(A+vv^{*}-\zeta I)^{-1}-(A-\zeta I)^{-1}]|\leq \frac{1}{|\Im(\zeta)|}.
		\end{align*} 
	\end{lem}
	
	\begin{proof}
		The proof is similar to Lemma 2.6 in \cite{silverstein1995empirical}. Using the resolvent identity and Lemma \ref{lem: Sherman-Morrison formula},
		\begin{align}\label{eqn: trace difference lemma estimate}
		&\tr[(A+vv^{*}-\zeta I)^{-1}-(A-\zeta I)^{-1}]\nonumber\\
		=&-\tr(A+vv^{*}-\zeta I)^{-1}vv^{*}(A-\zeta I)^{-1}\nonumber\\
		=&-v^{*}(A-\zeta I)^{-1}(A+vv^{*}-\zeta I)^{-1}v\nonumber\\
		=&-\frac{v^{*}(A-\zeta I)^{-1}(A-\zeta I)^{-1}v}{1+v^{*}(A-\zeta I)^{-1}v}.
		\end{align}
		
		Let $A=\sum_{i=1}^{n}\lambda_{i}(A)u_{i}u_{i}^{*}$ be the spectral decomposition of $A$, where $\lambda_{i}(A)\geq 0$ for all $1\leq i\leq n$. Then
		\begin{align}\label{eqn: lower bound of alpha_k}
		|v^{*}(A-\zeta I)^{-1}(A-\zeta I)^{-1}v|&=\sum_{i=1}^{n}\frac{|u_{i}^{*}v|^{2}}{|\lambda_{i}(A)-\zeta|^{2}},\nonumber\\
		|1+v^{*}(A-\zeta I)^{-1}v|^{2}&=\left|1+\sum_{i=1}^{n}\frac{|u_{i}^{*}v|^{2}}{\lambda_{i}(A)-\zeta}\right|^{2}\nonumber\\
		&=\left|1+\sum_{i=1}^{n}\frac{(\lambda_{i}(A)-\bar{\zeta})|u_{i}^{*}v|^{2}}{|\lambda_{i}(A)-\zeta|^{2}}\right|^{2}\nonumber\\
		%&=\left|1+\sum_{i=1}^{n}\frac{(\lambda_{i}(A)-\Re(\zeta))|u_{i}^{*}v|^{2}}{|\lambda_{i}(A)-\zeta|^{2}}+i\sum_{i=1}^{n}\frac{\Im(\zeta)|u_{i}^{*}v|^{2}}{|\lambda_{i}(A)-\zeta|^{2}}\right|^{2}\nonumber\\
		&=\left|1+\sum_{i=1}^{n}\frac{(\lambda_{i}(A)-\Re(\zeta))|u_{i}^{*}v|^{2}}{|\lambda_{i}(A)-\zeta|^{2}}\right|^{2}+\left|\sum_{i=1}^{n}\frac{\Im(\zeta)|u_{i}^{*}v|^{2}}{|\lambda_{i}(A)-\zeta|^{2}}\right|^{2}\nonumber\\
		%&\geq \left[(\Re(\zeta)\1_{\{\Re(\zeta)\leq 0\}})^{2}+(\Im(\zeta))^{2}\right]\left|\sum_{i=1}^{n}\frac{|u_{i}^{*}v|^{2}}{|\lambda_{i}(A)-\zeta|^{2}}\right|^{2}\nonumber\\
		&\geq |\Im(\zeta)|^{2}\left|\sum_{i=1}^{n}\frac{|u_{i}^{*}v|^{2}}{|\lambda_{i}(A)-\zeta|^{2}}\right|^{2}.
		\end{align}
		Plugging in the above estimates in \eqref{eqn: trace difference lemma estimate}, we obtain the result.
	\end{proof}
	
	\begin{lem}[Lemma 2.7 from \cite{MR1617051} and Equation (2.5) in \cite{rider2006gaussian}]\label{lem: rider lemma concentration around trace}
		Let $\xi=(\xi_{1},\xi_{2},\ldots, \xi_{n})$ be a random vector such that $\xi_{i}$ are iid complex valued random variables with $\E[\xi_{1}]=0$ and $\E[|\xi_{1}|^{2}]=1$. Then for any deterministic $n\times n$ matrix $A$,
		\begin{align*}
		&\E[|\xi^{*}A\xi-\tr A|^{p}]\leq C_{1}(p)((\E|\xi_{1}|^{4}\tr A^{*}A)^{p/2}+\E[|\xi_{1}|^{2p}]\tr (A^{*}A)^{p/2}),\\
		&\E[|\xi^{*}A\xi|^{p}]\leq C_{2}(p)\E[|\xi_{1}|^{2p}]((\tr A^{*}A)^{p/2}+|\tr A|^{p}),
		\end{align*}
		where $C_{1}(p), C_{2}(p)$ are constants that depend only on $p$.
	\end{lem}

	\begin{cor}\label{cor: estimates about band matrices}
		Let $I\subset\{1,2,\ldots, n\}$ be a fixed index set and $\xi_{1},\xi_{2},\ldots, \xi_{n}$ be a set of iid complex valued random variables with $\E[\xi_{1}]=0$ and $\E[|\xi_{1}|^{2}]=1$. Define $v=(v_{1},v_{2},\ldots, v_{n})$ where $v_{i}=\xi_{i}\1_{\{i\in I\}}$. Then for any fixed $n\times n$ deterministic matrix $A$ we have
		\begin{align*}
		\E\left[\left|v^{*}Av-\sum_{i\in I}a_{ii}\right|^{p}\right]\leq C(p)|I|^{p/2}\E[|\xi_{1}|^{2p}]\|A\|^{p}.
		\end{align*}
	\end{cor}
	\begin{proof}
		Let us define an $n\times n$ matrix $\tilde{A}$ as $(\tilde A)_{ij}=a_{ij}\1_{\{i\in I\}}\1_{\{j\in I\}}$, where $a_{ij}=(A)_{ij}$. Then, $v^{*}Av=v^{*}\tilde{A}v$. In addition, $\tr \tilde{A}=\sum_{i\in I}a_{ii}$. Therefore, using Lemma \ref{lem: rider lemma concentration around trace} and the fact that $\tr (\tilde{A}^{*}\tilde{A})\leq |I|\|\tilde A\|^{2}\leq |I|\|A\|^{2}$, the claim of the corollary follows.
	\end{proof}
	
	\begin{lem} \label{lem: rank one perturbation on matrices}
		Let $P$ and $Q$ be two $n\times n$ non-negative definite matrices, then for any $\zeta\in \C\backslash \R_{+}$ and $I\subset \{1,2,\ldots, n\}$,
		\begin{align*}
		\left|\sum_{k\in I}(P-\zeta I)^{-1}_{kk}-\sum_{i\in I}(Q-\zeta I)^{-1}_{kk}\right|\leq \frac{2}{|\Im(\zeta)|} \rank(P-Q).
		\end{align*}
	\end{lem}

\begin{proof}
	The above lemma is similar to Lemma C.3 from \cite{bordenave2013localization}. For the readers' convenience, we include the proof here. Using the resolvent identity, we have
	\begin{align*}
	(P-\zeta I)^{-1}-(Q-\zeta I)^{-1}=(P-\zeta I)^{-1}(Q-P)(Q-\zeta I)^{-1}.
	\end{align*}
	Therefore, $r:=\rank[(P-\zeta I)^{-1}-(Q-\zeta I)^{-1}]\leq \rank(P-Q)$. Let us write the singular value decomposition as
	\begin{align*}
	(P-\zeta I)^{-1}-(Q-\zeta I)^{-1}=\sum_{i=1}^{r}s_{i}u_{i}v_{i}^{*},
	\end{align*}
	where $s_{1},s_{2},\ldots,s_{r}$ are at most $r$ non zero singular values of $(P-\zeta I)^{-1}-(Q-\zeta I)^{-1}$, and $\{u_{1},u_{2},\ldots,u_{r}\}$, $\{v_{1},v_{2},\ldots,v_{r}\}$ are two sets of orthonormal vectors. Consequently, we may write
	\begin{align*}
	(P-\zeta I)^{-1}_{kk}-(Q-\zeta I)^{-1}_{kk}=\sum_{i=1}^{r}s_{i}(e_{k}^{T}u_{i})(v_{i}^{*}e_{k}).
	\end{align*}
	
	Using Cauchy-Schwarz inequality,
	\begin{align*}
	\left|\sum_{k\in I}(P-\zeta I)^{-1}_{kk}-\sum_{k\in I}(Q-\zeta I)^{-1}_{kk}\right|&\leq \sum_{i=1}^{r}s_{i}\sum_{k\in I}|e^{T}u_{i}||v_{i}^{*}e_{k}|\\
	&\leq \sum_{i=1}^{r}s_{i}\sqrt{\sum_{k\in I}|e_{k}^{T}u_{i}|^{2}}\sqrt{\sum_{k\in I}|v_{k}^{*}e_{k}|^{2}}\\
	&\leq \sum_{i=1}^{r}s_{i}\|u\|\|v\|\\
	&\leq \sum_{i=1}^{r}s_{i}\leq \frac{2r}{|\Im(\zeta)|}\leq \frac{2}{|\Im(\zeta)|}\rank(P-Q),
	\end{align*}
	where the second last inequality follows from the fact that $s_{i}\leq \|(P-\zeta I)^{-1}-(Q-\zeta I)^{-1}\|\leq 2/|\Im(\zeta)|$ for all $1\leq i\leq r$.
\end{proof}
	
	\begin{result}[Azuma-Hoeffding inequality; see \cite{MR1036755}]\label{lem: azuma inequality}
		Let $\{\xi_{k}\}_{k}$ be a martingale with respect to the filtration $\{\CF_{k}\}_{k}$ such that for all $k$, $|\xi_{k+1}-\xi_{k}|\leq c_{k}$ almost surely. Then for any $t>0$
		\begin{align*}
		\Pb(|\xi_{n}-\E[\xi_{n}]|>t)\leq 2\exp\left\{-\frac{t^{2}}{2\sum_{k=1}^{n}c_{k}^{2}}\right\}.
		\end{align*}
	\end{result}
A simple consequence of the previous concentration inequality is a bound on the moments.
\begin{cor} \label{cor: azumamoment}
	Under the conditions of Result \ref{lem: azuma inequality}, for $l \in \N$, we have
	\[
	\E[ |\xi_n - \E \xi_n|^l] \leq C(l) \left(\sum_{k=1}^n c_k^2 \right)^{l/2}
	\]
	where $C(l)$ is a constant only depending on $l$.
\end{cor}

\begin{proof}
	This result can be deducued from the straightforward calculation using Result \ref{lem: azuma inequality},
	\begin{align*}
	\E[ |\xi_n - \E \xi_n|^l] &= l \int_0^\infty t^{l-1} \Pb(|\xi_n - \E \xi_n| > t) \, dt \\
	&\leq 2l \int_0^\infty t^{l-1} \exp\left( -\frac{t^2}{2 \sum_{k=1}^n c_k^2} \right) \, dt \\
	&= l \left(2 \sum_{k=1}^n c_k^2 \right)^{l/2} \int_0^\infty u^{l/2-1} e^{-u} \, du \\
	&= l \Gamma(l/2) 2^{l/2} \left(2 \sum_{k=1}^n c_k^2 \right)^{l/2},
	\end{align*}
	where $\Gamma$ is the gamma function. 
\end{proof}
	
	Our final lemma is a technical observation which is of use in Section \ref{sec:esd}.
	\begin{lem} \label{lem:diagonalexpectation}
		We let $X$ be the random matrix from Theorem \ref{thm: main theorem} (without the restriction on the bandwidth).
		We recall the notation from Section \ref{sec:esd}.  For fixed $z \in \mathbb{C}$ and $\zeta$ in the upper half of the complex plane,
		\[
		P_{z,\zeta}:=(X_{z}X_{z}^{*})_{\zeta}=(X-zI)(X-zI)^{*}-\zeta I.
		\]
		Then for all $1\leq i\leq n$, 
		\[
		\E[\{P_{z,\zeta}\}^{-1}_{ii}]=\E[\{P_{z,\zeta}\}^{-1}_{11}].
		\]
	\end{lem}

	\begin{proof}
	%It suffices to establish this assuming $\zeta = 0$.  
	We divide $[n]$ into sets $I_1, \dots, I_m$ where $I_i = [(i-1)b_n+1, i b_n] \cap \mathbb{N}$.  Let $P_{ij}$ denote the $n \times n$ permutation matrix that permutes the $i$-th and $j$-th column when acting from the left on a matrix.  Observe that when $i,j \in I_k$ for some $k \in [m]$, $P_{ij} X_z P_{ij}^{-1}$ has the same distribution as $X_z$ due to the iid assumption and block structure.  Therefore, $X_z X_z^\ast$ has the same distribution as $P_{ij} X_z P_{ij}^T P_{ij} X^*_z P_{ij}^T = P_{ij} X_zX^*_z P_{ij}^T $.  Thus,
	\[
	(X_z X_z^* - \zeta I)^{-1}_{ii} \sim (P_{ij} (X_z X_z^* - \zeta I) P_{ij}^T)^{-1}_{ii} = (P_{ij} (X_z X_z^* - \zeta I)^{-1} P_{ij})_{ii} \sim 	(X_z X_z^* - \zeta I)^{-1}_{jj}, 
	\]     
	where we use $\sim$ to denote equality in distribution.  This establishes that the expectation for any two indices in the same index block are identical.  It remains to show that the expectations for the various blocks are the same.  Here, we define a permutation that exploits the block-band structure.  Let $P$ be the permutation that cyclically shifts $I_k$ to $I_{k+1}$ maintaining the order within each block and using the convention that $I_{m+1}  = I_1$.  By the structure of the matrix and the iid assumption, 
	\[
	X_z X_z^* \sim P X_z X_z^* P^{-1}.
	\]
	Thus, 
	\begin{align*}
	(X_z X^*_z - \zeta I)^{-1}_{11} &\sim (P (X_z X^*_z - \zeta I) P^{-1})^{-1}_{11} \\
	&= (P^{-1} (X_z X_z^* - \zeta I)^{-1} P)_{11} \\
	&\sim (X_z X_z^* - \zeta I)^{-1}_{b+1, b+1}.
	\end{align*}
	Continuing inductively establishes the equivalence of all the expectations along the diagonal of $(X_z X^*_z - \zeta I)^{-1}$.
	\end{proof}
	
	\section*{Data availability statement}
	Data sharing is not applicable to this article as no new data were created or analyzed in this study.
	
	\section*{Acknowledgment}
	The authors thank the anonymous referees for useful feedback and corrections.  
	K. Luh has been supported in part by NSF grant DMS-1702533.  S. O'Rourke has been supported in part by NSF grants ECCS-1913131 and DMS-1810500.

	\bibliographystyle{abbrv}
	\bibliography{MasterBib}
\end{document}